\documentclass[10pt, letterpaper,reqno]{amsart}
\usepackage[left=1in,right=1in,bottom=0.7in,top=0.7in]{geometry}
\usepackage{amsfonts}
\usepackage{amsmath, amssymb}
\usepackage{graphicx}
\usepackage[font=small,labelfont=bf]{caption}
\usepackage{epstopdf}
\usepackage{xcolor}
\usepackage{amsthm}
\usepackage{float}
\usepackage{enumitem}
\usepackage{pgfplots}
\usepackage{listings}
\usepackage{longtable}
\usepackage{mathrsfs}
\usepackage{dsfont}
\usepackage{mathtools}
\usepackage{dsfont}
\usepackage{stmaryrd}
\usepackage[pdfpagelabels,hyperindex]{hyperref}
\hypersetup{linkbordercolor=green}

\newcommand*{\rom}[1]{\expandafter\@slowromancap\romannumeral #1@}
 \colorlet{lgray}{white!80!black}
\colorlet{lred}{white!85!red}
\colorlet{lgreen}{white!60!green}
\colorlet{dgreen}{black!30!green}
\colorlet{lpurple}{white!60!purple}
\colorlet{lblue}{white!60!blue}
\definecolor{green}{rgb}{0.1,0.8,0.1}
\definecolor{yellow}{rgb}{1.0,0.85,0.25}
\definecolor{purple}{rgb}{1.0, 0, 1.0}
\definecolor{blue}{rgb}{0, 0, 1.0}
\tikzstyle{unfused}=[lgray, line width=1.5pt, ->]
\tikzstyle{fused}=[lgray, line width=4pt, ->]
\tikzstyle{dual}=[black, line width=1pt, dashed]
\tikzstyle{lightdual}=[black, line width=0.5pt, dashed]
\tikzstyle{cut}=[black, line width=1.0pt]

\theoremstyle{plain}
\newtheorem{theorem}{Theorem}[section]
\newtheorem{lemma}[theorem]{Lemma}
\newtheorem{proposition}[theorem]{Proposition}

\theoremstyle{remark}
\newtheorem{definition}[theorem]{Definition}

\colorlet{shadecolor}{gray!20}
\pgfplotsset{compat=1.9}
\usetikzlibrary{shapes.multipart}
\usetikzlibrary{patterns}
\usetikzlibrary{shapes.multipart}
\usetikzlibrary{arrows}
\usetikzlibrary{decorations.markings}
\usepgflibrary{decorations.shapes}
\usetikzlibrary{decorations.shapes}
\usepgflibrary{shapes.symbols}
\usetikzlibrary{shapes.symbols}
\usetikzlibrary{decorations.pathreplacing}
\tikzstyle{fleche}=[>=stealth', postaction={decorate}, thick]
\tikzstyle{axis}=[->, >=stealth', thick, gray]
\tikzstyle{paths}=[>->, >=stealth', thick]
\tikzstyle{path}=[->, >=stealth', thick]
\tikzstyle{grille}=[dotted, gray]
\DeclareFontFamily{U}{mathx}{}
\DeclareFontShape{U}{mathx}{m}{n}{<-> mathx10}{}
\DeclareSymbolFont{mathx}{U}{mathx}{m}{n}
\DeclareMathAccent{\widehat}{0}{mathx}{"70}
\DeclareMathAccent{\widecheck}{0}{mathx}{"71}

\usepackage{romannum}
\usepackage{upgreek}

\usepackage[scr=boondoxo]{mathalfa}

\newcommand{\lrw}{\mathsf{LRW}}
\newcommand{\bg}{\boldsymbol{\gamma}}
\newcommand{\bs}{\boldsymbol{\sigma}} 

\renewcommand{\a}{\mathsf{a}}
\renewcommand{\b}{\mathsf{b}}

\usepgflibrary{fpu}
\makeatletter
\let\NAT@parse\undefined
\makeatother

\author{Dominik Schmid}
\address
{
Dominik Schmid, 
University of Augsburg, Germany 
}
\email{d.schmid@uni-a.de} 
\author{Zongrui Yang}
\address
{
Zongrui Yang, 
University of California, Berkeley
}
\email{zongrui@berkeley.edu}

\title[Stationary measures of the open KPZ equation]{Stationary measures of the open KPZ equation: A short proof}  
\begin{document}
\pagenumbering{arabic}

\begin{abstract}  
In this note, we give a short probabilistic proof of the two-layer Brownian representation of the stationary measure of the open Kardar--Parisi--Zhang (KPZ) equation on a finite interval with arbitrary Neumann boundary parameters, conjectured by Barraquand and Le Doussal and recently established in full generality by Dunlap, Gu and Rosati. Our starting point is Bryc's two-layer Gibbs representation of the stationary measure of the open asymmetric simple exclusion process (ASEP) under Liggett's condition, together with our previously established probabilistic interpretation of the associated Bryc polynomials in terms of inhomogeneous birth--death Markov chains. Martingale estimates yield uniform asymptotics for these polynomials under weakly asymmetric scaling. Combined with known convergence results from open ASEP to open KPZ, these estimates recover the stationary height profile of the open KPZ equation as the first marginal of a reweighted pair of Brownian motions.
\end{abstract}
\maketitle

%
\section{Introduction and main results}

\subsection{Preface}

The Kardar--Parisi--Zhang (KPZ) equation, introduced in \cite{Kardar1986Dynamic}, describes the growth of random interfaces. On a finite interval with Neumann boundary conditions, it is known as the open KPZ equation \cite{corwin2018open}. A basic problem is to identify its stationary measure and understand how it depends on the boundary parameters. Here, stationarity concerns the height function relative to its value at the left endpoint of the interval. We refer to \cite{corwin2022some} for a survey.

The open KPZ equation arises as a weakly asymmetric scaling limit of the open asymmetric simple exclusion process (ASEP), as established by Corwin and Shen \cite{corwin2018open} and extended to general boundary parameters by Parekh \cite{parekh2017kpz}; see \cite{bertini1997stochastic} for the classical full-space result. This connection provides a route to its stationary measure through the stationary measures of the open ASEP. Corwin and Knizel \cite{corwin2021stationary} first characterized the stationary measures of the open KPZ equation through their multipoint Laplace transforms in the fan region, i.e., when the sum of the boundary parameters is nonnegative. \cite{bryc2021markov,barraquand2021steady} inverted these Laplace transforms and obtained the reweighted two-layer Brownian representation in the fan region. \cite{barraquand2023stationary,himwich2024stationary} also developed a separate derivation in the fan region through reweighted random walks.

Barraquand and Le Doussal \cite{barraquand2021steady} conjectured that the Brownian representation remains valid for all real boundary parameters. Barraquand, Corwin, and the second author of this paper \cite{barraquand2024stationary} proved this conjecture modulo an expected convergence from the log-gamma polymer on a strip to the open KPZ equation.
Recently, Dunlap, Gu and Rosati \cite{DGR:openKPZ} proved this conjecture by stochastic analytic methods. Together with the uniqueness of the stationary measure established in \cite{parekh2022ergodicity,knizel2022strong}, this gives a complete characterization. Very recently, Himwich and Wang \cite{himwich2026laplace} also obtained Laplace-transform formulas for a range of parameters in the shock region by analytic continuation.

In this note, we give a short probabilistic proof of the two-layer Brownian representation for every interval length $L>0$ and all real boundary parameters. Our starting point is Bryc's two-layer Gibbs representation of the stationary measure of the open ASEP under Liggett's condition \cite{bryc2024stationary}, together with our probabilistic interpretation of the associated Bryc polynomials in terms of inhomogeneous birth--death chains \cite{SchmidYangShock}. Martingale estimates lead to uniform asymptotics for these polynomials under weakly asymmetric scaling. An invariance principle and uniform integrability then identify the limiting two-layer measure as the law of a reweighted pair of Brownian motions. Combining with the convergence results from open ASEP to open KPZ \cite{corwin2018open,parekh2017kpz}, we are able to show that its first marginal is stationary for the open KPZ equation. We point out that, although the discrete approximation satisfies Liggett's condition, the scaling limit covers all real KPZ boundary parameters.

\subsection{The open KPZ equation and its stationary measure}\label{subsec:model and results}

Let us start by formally describing the open KPZ equation. 
For $L>0$, the open KPZ equation on $[0,L]$ with inhomogeneous Neumann boundary conditions and parameters $u,v\in \mathbb R$ is the stochastic PDE 
\begin{equation}
\begin{cases} 
    \partial_t h(t,x) = \frac{1}{2} \partial_{xx} h(t,x) + \frac{1}{2} \left( \partial_x h(t,x)\right)^2 + \xi(t,x), \\ 
    \partial_x h(t,x)\big\vert_{x=0} = u, \\ 
    \partial_x h(t,x)\big\vert_{x=L} = -v, 
    \end{cases}
    \tag{KPZ$_{u,v}$}
    \label{eq:KPZ}
\end{equation}
where $h$ is a function of  $t\in \mathbb R_+$ and $x\in [0,L]$, and $\xi(t,x)$ is a space-time white noise. 
In the following, let $C(E,F)$ denote the space of continuous functions from $E$ to $F$, equipped with the uniform topology. We call the law of a random function $h_0\in C([0,L],\mathbb R)$ with $h_0(0)=0$ the stationary measure for the open KPZ equation on $[0,L]$ if the solution to \eqref{eq:KPZ} with initial condition $h(0,x)=h_0(x)$ is such that the law of $h_t:=\big(h(t,x)-h(t,0)\big)_{x\in [0,L]}$ is the same for all $t\geq 0$. 

Let us recall the definition of a Hopf-Cole solution to the open KPZ equation; see, for example, \cite{corwin2018open}. 
\begin{definition}
\label{def:Hopf-Cole transform}
 Fix boundary parameters $u,v\in \mathbb R$. A stochastic process $h(t,x)\in C(\mathbb R_{\geq 0}\times [0,L], \mathbb R)$  is called a (Hopf-Cole) solution to the open KPZ equation \eqref{eq:KPZ} with initial condition $h(0,x)= h_0(x)$ if $h(t,x)=\log Z(t,x)$, where $Z(t,x)$ is the solution of the  stochastic heat equation with multiplicative noise
\begin{equation}
\begin{cases} 
    \partial_t Z(t,x) = \frac{1}{2} \partial_{xx} Z(t,x) +  Z(t,x) \xi(t,x), \\ 
    \partial_x Z(t,x)\Big\vert_{x=0} = \mu Z(t,0), \\ 
    \partial_x h(t,x)\Big\vert_{x=L} = \nu Z(t,0), 
    \end{cases}
    \tag{SHE$_{u,v}$}
    \label{eq:SHE}
\end{equation}
with $\mu=u-1/2$, $\nu=-(v-1/2)$ and  initial condition $Z(0,x)=e^{h_0(x)}$.  A stochastic process $Z(t,x)\in C(\mathbb R_{\geq 0}\times [0,L], \mathbb R_{>0})$ is called a solution to \eqref{eq:SHE} if for all $t> 0$ and $x\in [0,L]$, it is measurable with respect to the filtration of the space-time white noise  $\xi$ generated up to and including time $t$ and if it satisfies
\begin{equation}
    Z(t,x) = \int_{0}^L p_{\mu, \nu}(0,y; t,x)Z(0,y) d y + \int_0^t \int_0^L  p_{\mu, \nu} (s,y; t,x) Z(s,y)\xi(d s, d y),
\end{equation}
where the last integral is a stochastic integral in time in the Itô sense. Here, $p_{\mu, \nu}(s,y; t,x)$ is the heat kernel on $[0,L]$ with Robin type boundary condition, i.e.,~the unique solution to 
\begin{equation}\label{eq:heat kernel}
    \begin{cases}
    \partial_t p_{\mu, \nu}(s,y; t,x) = \frac 1 2 \partial_{xx} p_{\mu, \nu}(s,y; t,x),\\
    \partial_x p_{\mu, \nu}(s,y; t,x) \Big\vert_{x=0}=\mu p_{\mu, \nu}(s,y; t,0) \text{ for all }s<t, y\in [0,L],\\ 
        \partial_x p_{\mu, \nu}(s,y; t,x) \Big\vert_{x=L}=\nu p_{\mu, \nu}(s,y; t,L) \text{ for all }s<t, y\in [0,L],\\
        \lim_{t\to s}  p_{\mu, \nu}(s,y; t,x)=\delta(x-y) \text{ in the weak sense on }\mathbb L^2([0,L]). 
    \end{cases}
\end{equation}
\end{definition}

It is shown by \cite[Proposition 2.7]{corwin2018open} and \cite[Proposition 4.2]{parekh2017kpz} that, for all $u,v\in\mathbb R$, the stochastic heat equation \eqref{eq:SHE} admits a unique solution that is almost surely positive, subject to a finite second moment condition. In particular, its logarithm is finite, hence giving indeed a solution of the open KPZ equation \eqref{eq:KPZ}. Next, we introduce the Hariya-Yor process from \cite{HariyaYor2004}, which allows us to describe the stationary measures of the open KPZ equation.

\begin{definition} \label{def:KPZstat}
Fix \(L>0\).  For each $v_1,v_2\in \mathbb R$, let  $\mathbb P_{\mathrm{BM}}^{v_1,v_2}$  be the probability measure on $C([0,L], \mathbb R)^2$  of  two independent standard Brownian motions $B_1,B_2$ on $[0,L]$ with drifts $v_1$ and $v_2$ respectively.
For $u,v\in \mathbb R$ define the probability measure  $\mathbb P_{\mathrm{KPZ}}^{u,v}$  on $C([0,L], \mathbb R)$ through its Radon-Nikodym derivative by 
\begin{equation}
\label{eq:PstatKPZ}
\frac{d \mathbb P_{\mathrm{KPZ}}^{u,v}}{d\mathbb P_{\mathrm{BM}}^{0,0} }(B_1,B_2) =  \frac{1}{\mathcal Z_{\mathrm{KPZ}}^{u,v}}\left( \int_0^L d s e^{-(B_1(s)-B_2(s))}\right)^{-u-v}e^{-v(B_1(L)-B_2(L))}. 
\end{equation}
By the Cameron-Martin theorem, we may also include drifts in the reference measure and define $ \mathbb P_{\mathrm{KPZ}}^{u,v}$ as 
\begin{equation}
\label{eq:PstatKPZbis}
\frac{d \mathbb P_{\mathrm{KPZ}}^{u,v}}{d\mathbb P_{\mathrm{BM}}^{-v,v}}(B_1,B_2) =  \frac{e^{v^2L}}{\mathcal Z_{\mathrm{KPZ}}^{u,v}}\left( \int_0^L d s e^{-(B_1(s)-B_2(s))}\right)^{-u-v}. 
\end{equation}
\end{definition}

 We have the following characterization of the stationary measure of the open KPZ equation \eqref{eq:KPZ}.

\begin{theorem}[c.f. Theorem 1.1 in \cite{DGR:openKPZ}]\label{thm:main open KPZ stat}
    For any $L>0$ and $u,v\in \mathbb R$, the law of $B_1$ under $\mathbb P_{\mathrm{KPZ}}^{u,v}$ is the (unique ergodic) stationary measure for \eqref{eq:KPZ}. 
\end{theorem} 

Theorem \ref{thm:main open KPZ stat} was rigorously established in \cite{DGR:openKPZ} for full parameters by stochastic analytic methods, resolving the conjecture from \cite{barraquand2021steady}.
Our aim is to give a short probabilistic proof of Theorem~\ref{thm:main open KPZ stat}. To this end, we analyze the stationary measure
of the open ASEP under weakly asymmetric scaling.

\subsection{The open ASEP}

Solving the (open) KPZ equation is closely related to studying the asymmetric simple exclusion process; see \cite{bertini1997stochastic} for a first construction of solutions to the full space KPZ equation, and \cite{corwin2018open,parekh2017kpz} for open KPZ equation solutions arising from the open ASEP. Let us now formally describe the open ASEP and the relevant weakly asymmetric scaling.

The open ASEP is a continuous-time Markov chain on $\{0,1\}^{N}$  with parameters
\begin{equation}\label{eq:conditions open ASEP}
\alpha,\beta>0,\quad \gamma,\delta\geq 0,\quad\text{and}\quad  0\leq q<1 . 
\end{equation}
The state of the system is represented by a configuration $(\tau_1,\tau_2,\dots,\tau_N) \in \{0,1\}^N$ where $\tau_i$ indicates whether the $i^{\text{th}}$ site is occupied or empty. In the bulk of the system, each particle attempts to jump to the left according to rate $q$ Poisson clocks, and to the right at rate $1$ Poisson clocks, independently of all other particles. An attempt fails when
the target site is already occupied.  At the two boundaries, particles can enter and exit the segment. More precisely, at the left boundary, particles attempt to enter the system from the
reservoir at random with rate $\alpha$ and attempt to exit at random with rate $\gamma$. Similarly, at the right boundary, particles
attempt to enter at random with rate $\delta$ and attempt to exit at rate $\beta$; see Figure~\ref{fig:openASEP} for an illustration when the rates are allowed to depend on the system size. The assumption $q<1$ ensures that  particles have a  higher rate of moving to the right.  

Next, we define the boundary parameters
\begin{equation}\label{eq:defining ABCD}
\a=\phi_+(\alpha,\gamma),\quad\b=\phi_+(\beta,\delta),\quad \mathsf{c}=\phi_-(\alpha,\gamma), \quad \mathsf{d}=\phi_-(\beta,\delta),
\end{equation}
where we have
\begin{equation} \label{eq:phi}
\phi_{\pm}(x,y):=\frac{1}{2x}\left(1-q-x+y\pm\sqrt{(1-q-x+y)^2+4xy}\right), \quad \mbox{for }\; x>0\mbox{ and }y\geq 0.
\end{equation}
In particular, note that 
\begin{equation}\label{eq:conditions qABCD}
\a,\b\geq0,\quad -1<\mathsf{c},\mathsf{d}\leq 0,\quad 0\leq q<1.
\end{equation}
Intuitively, $(1+\a)^{-1}$, respectively $\b(1+\b)^{-1}$, can be interpreted as the effective density at the left side, respectively right side, of the segment.
Moreover, for the parameters $\gamma$ and $\delta$, we let
\begin{equation} \label{eq:Liggett's condition}
\gamma=q(1-\alpha),\quad\delta=q(1-\beta),\quad \alpha,\beta\in(0,1].
\end{equation}
which is known as Liggett's condition; see \cite{liggett1975ergodic}. In this case, the above parameters simplify to 
\begin{equation}\label{eq:parameters under Liggett condition}
\a=\frac{1-\alpha}{\alpha},\quad\b=\frac{1-\beta}{\beta},\quad \mathsf{c}=-q,\quad\mathsf{d}=-q , 
\end{equation} and effective densities are given by $\alpha$ and $1-\beta$, respectively. 
We will rely on Liggett's condition in Section~\ref{sec:BrycPoly} when we characterize the stationary distribution of the open ASEP using Bryc polynomials. 

\begin{figure}
\centering
\begin{tikzpicture}[scale=0.85]
\def\spiral[#1](#2)(#3:#4:#5){
\pgfmathsetmacro{\domain}{pi*#3/180+#4*2*pi}
\draw [#1,
       shift={(#2)},
       domain=0:\domain,
       variable=\t,
       smooth,
       samples=int(\domain/0.08)] plot ({\t r}: {#5*\t/\domain})
}
\def\particles(#1)(#2){
  \draw[black,thick](-3.9+#1,0.55-0.075+#2) -- (-4.9+#1,0.55-0.075+#2) -- (-4.9+#1,-0.4-0.075+#2) -- (-3.9+#1,-0.4-0.075+#2) -- (-3.9+#1,0.55-0.075+#2);
\node[shape=circle,scale=0.6,fill=red] (Y1) at (-4.15+#1,0.2-0.075+#2) {};
  	\node[shape=circle,scale=0.6,fill=red] (Y2) at (-4.6+#1,0.35-0.075+#2) {};
  	\node[shape=circle,scale=0.6,fill=red] (Y3) at (-4.2+#1,-0.2-0.075+#2) {};
   	\node[shape=circle,scale=0.6,fill=red] (Y4) at (-4.45+#1,0.05-0.075+#2) {};
  	\node[shape=circle,scale=0.6,fill=red] (Y5) at (-4.65+#1,-0.15-0.075+#2) {}; }
\def\annhil(#1)(#2){	  \spiral[black,thick](9.0+#1,0.09+#2)(0:3:0.42);
\draw[black,thick](8.5+#1,0.55+#2) -- (9.5+#1,0.55+#2) -- (9.5+#1,-0.4+#2) -- (8.5+#1,-0.4+#2) -- (8.5+#1,0.55+#2); }
\node[shape=circle,scale=1.5,draw] (B) at (2.3,0){} ;
	\node[shape=circle,scale=1.5,draw] (C) at (4.6,0) {};
	\node[shape=circle,scale=1.2,fill=red] (CB) at (2.3,0) {};
    \node[shape=circle,scale=1.5,draw] (A) at (0,0){} ;
 	\node[shape=circle,scale=1.5,draw] (D) at (6.9,0){} ; 
 	 	\node[shape=circle,scale=1.5,draw] (Z) at (-2.3,0){} ;
	\node[shape=circle,scale=1.2,fill=red] (YZ) at (-2.3,0) {};
   \node[line width=0pt,shape=circle,scale=1.6] (B2) at (2.3,0){};
	\node[line width=0pt,shape=circle,scale=2.5] (D2) at (6.9,0){};
		\node[line width=0pt,shape=circle,scale=2.5] (Z2) at (-2.3,0){};
		
	\node[line width=0pt,shape=circle,scale=2.5] (X10) at (6.8,0){};
		\node[line width=0pt,shape=circle,scale=2.5] (X11) at (-2.2,0){};	
			\node[line width=0pt,shape=circle,scale=2.5] (X12) at (8.4,0){};
		\node[line width=0pt,shape=circle,scale=2.5] (X13) at (-3.8,0){};		
\draw[thick] (Z) to (A);	
	\draw[thick] (A) to (B);	
		\draw[thick] (B) to (C);	
  \draw[thick] (C) to (D);

\particles(0)(0);
\particles(6.9+4.9+1.6)(0);


\draw [->,line width=1pt]  (B2) to [bend right,in=135,out=45] (C);
 \draw [->,line width=1pt] (6.9+1.6,-0.475) to [bend right,in=135,out=45] (D);	
   \draw [->,line width=1pt] (B2) to [bend right,in=-135,out=-45] (A);
     \draw [->,line width=1pt] (Z) to [bend right,in=135,out=45] (-3.9,-0.475);
    \node (text1) at (3.5,1){$1$} ;    
	\node (text2) at (1.1,1){$q_n$} ;   
	\node (text3) at (-2.3-1,1){$\alpha_n$}; 
	\node (text4) at (6.9+1,1){$\beta_n$};   
    \node (text3) at (-2.3-1,-1){$\gamma_n$}; 
	\node (text4) at (6.9+1,-1){$\delta_n$};   

    \node[scale=0.9] (text1) at (-2.2,-0.7){$1$} ;    
    \node[scale=0.9] (text1) at (6.8,-0.7){$n$} ;   
  	
  \draw [->,line width=1pt] (-3.9,0.475) to [bend right,in=135,out=45] (Z);
   \draw [->,line width=1pt] (D) to [bend right,in=135,out=45] (6.9+1.6,0.475);

	\end{tikzpicture}	
\caption{\label{fig:openASEP} Visualization of the open ASEP under weakly asymmetric scaling for parameters $(q_n,\alpha_n,\beta_n,\gamma_n,\delta_n)$ and $L=1$ on a segment of length $n$.}
\end{figure}
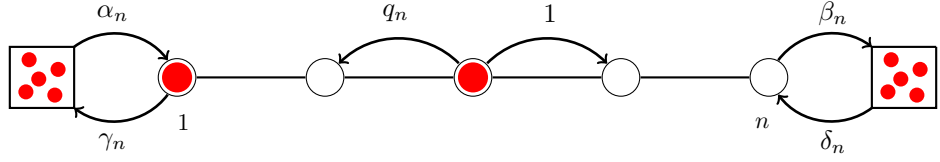
Note that for all $\alpha,\beta>0$, the open ASEP on a segment of length $n$ has a unique stationary distribution $\mu^{(q,\a,\b)}_n$ for all $n\in \mathbb N$. For fixed $L>0$ and for $n\in \mathbb N$, we consider the open ASEP on a segment $\llbracket 1,\lfloor L n \rfloor \rrbracket$, where we write $\llbracket a,b\rrbracket := \mathbb Z \cap [a,b]$ for all $a,b \in \mathbb Z$ with $a < b$. The following is our main result on the stationary distribution of the open ASEP under weakly asymmetric scaling and Liggett's condition. 

\begin{theorem}\label{thm:open KPZ equation scaling}
Fix \(L>0\) and \(u,v\in\mathbb R\). For \(n\in\mathbb N\), let
\begin{equation}\label{eq:WASEP limit}
q_n=e^{-2/\sqrt n},\qquad
\alpha_n=\frac{1}{1+q_n^u},
\qquad
\beta_n=\frac{1}{1+q_n^v},
\qquad
\gamma_n=q_n(1-\alpha_n),
\qquad
\delta_n=q_n(1-\beta_n).
\end{equation}
For all sufficiently large \(n\), so that \(\lfloor nL\rfloor\ge1\), consider the open ASEP on the lattice $\llbracket1,\lfloor nL\rfloor\rrbracket$ with jump rates given by $\alpha_n,\beta_n,\gamma_n,\delta_n,q_n$. Let $\tau=(\tau_1,\ldots,\tau_{\lfloor nL\rfloor})\in\{0,1\}^{\lfloor nL\rfloor}$ be distributed according to the stationary measure of this model, and let 
\begin{equation}\label{eq:jtynd}
h_n^\tau(x)=\frac{1}{\sqrt n}\sum_{i=1}^{\lfloor nx\rfloor}(2\tau_i-1),\quad\mbox{ for } x\in[0,L].
\end{equation}
We view \(h_n^\tau\) as a random element of \(D([0,L],\mathbb R)\) equipped with the Skorokhod \(J_1\) topology.  Then, as \(n\to\infty\),
$h_n^\tau\Rightarrow B_1$ weakly in \(D([0,L],\mathbb R)\), where \(B_1\) is the first coordinate under \(\mathbb P_{\mathrm{KPZ}}^{u,v}\) from Definition \ref{def:KPZstat}.
\end{theorem}
 
The main goal of this note is to provide a short proof of Theorem~\ref{thm:open KPZ equation scaling} which in return gives Theorem~\ref{thm:main open KPZ stat}. 

\begin{proof}[Proof of Theorem \ref{thm:main open KPZ stat} assuming Theorem \ref{thm:open KPZ equation scaling}]
    Under the weakly asymmetric limit of parameters \eqref{eq:WASEP limit},  it is shown in \cite[Theorem 1.2 (1)]{corwin2021stationary} for $L=1$ and \cite[Proposition~2.13, Theorem~2.14]{himwich2024stationary} for general $L>0$ that, as $n\to\infty$, the (rescaled) height function processes \eqref{eq:jtynd} for the stationary measures of open ASEP are tight and converge to the stationary measure of the open KPZ equation. We mention that this convergence of stationary measures follows from the process level convergence from the (rescaled) open ASEP to the open KPZ equation obtained in \cite{corwin2018open} for $u,v\geq \frac{1}{2}$ and \cite{parekh2017kpz} for general $u,v\in\mathbb R$. 
    With this input, the description of the stationary measure of the open KPZ equation stated in Theorem \ref{thm:main open KPZ stat} follows directly from  Theorem \ref{thm:open KPZ equation scaling} as a scaling limit of the open ASEP.
    Finally, let us stress that the ergodicity of the open KPZ equation and uniqueness of its stationary measure follow directly from \cite{knizel2022strong, parekh2022ergodicity} for $L=1$, while the results in \cite{parekh2022ergodicity} directly extend to general $L>0$; see also the comment after \cite[Theorem 1.4]{himwich2024stationary}.  
\end{proof}

\subsection{Related works}

The open KPZ equation can be studied from various perspectives; see \cite{corwin2022some} for a survey. The stationary measures of the open KPZ equation were first characterized by Corwin and Knizel in \cite{corwin2021stationary} when $u+v \geq 0$ and $L=1$  through its multipoint Laplace transform. Their analysis uses weakly asymmetric limits of the Askey--Wilson process formulas for the stationary measure of the open ASEP. Barraquand and Le Doussal \cite{barraquand2021steady} and Bryc, Kuznetsov, Wang and
Weso{\l}owski \cite{bryc2021markov} obtained probabilistic representations of the stationary measures by two-layer Brownian motions through inverting these Laplace transforms. The former conjectured the Brownian representation for arbitrary interval lengths and all real boundary parameters, while the latter rigorously identified the stationary measure when $u+v\geq0$, $\min(u,v)>-1$, and $L=1$. In the fan region $u+v\geq0$, a direct derivation of the Brownian representation (without going through the Laplace transforms) was given by Barraquand and Le Doussal \cite{barraquand2023stationary} using the Enaud--Derrida representation of the matrix product ansatz for the open ASEP. The required weak convergence was subsequently established rigorously by Himwich \cite{himwich2024stationary}. 

Barraquand, Corwin and the second author of this paper \cite{barraquand2024stationary} developed a two-layer Gibbs framework for stationary measures of integrable models with two-sided open boundaries. Through an intermediate-disorder limit of the log-gamma polymer stationary measure on a strip, they obtained the Brownian formula for all boundary parameters, conditional on the corresponding (conjectured) convergence of the log-gamma polymer model on a strip to the open KPZ equation. This convergence has been established rigorously in the full space \cite{alberts2014intermediate} and the half-space \cite{wu2018intermediate,parekh2019positive,barraquand2022stationary}; the convergence in the two-boundary case was conjectured but not shown. Recently, Dunlap, Gu and Rosati \cite{DGR:openKPZ} established the Brownian
formula for all interval lengths and boundary parameters using stochastic analytic methods. Very recently, Himwich and Wang \cite{himwich2026laplace} extended the Laplace-transform description to a range of parameters in the shock region through analytic continuation. Uniqueness and ergodicity of the stationary measure of the open KPZ equation was addressed in \cite{knizel2022strong,parekh2022ergodicity}, building on earlier ideas from \cite{hairer2018strong}. In particular, it was shown that the open KPZ equation exhibits a unique stationary measure.     
Let us also mention recent progress on the growth rate \cite{barraquand2024integral}, cumulants \cite{Barraquand_2025}, and fluctuations \cite{hip2025fluctuation} of the open KPZ equation.
In a similar direction, stationary measures for the half-space KPZ equation, i.e.~the KPZ equation in the presence of one open boundary, were established in \cite{barraquand2022stationary}, and the limits from the stationary measures of the open KPZ equation to that of the half-space KPZ equation was studied in \cite{bryc2021markov2}.

As we have mentioned in the preface, the open KPZ equation can be approximated by the discrete model of open ASEP.
The structure of stationary measures for the open ASEP has been a topic of main interest over the past decades; see \cite{liggett1975ergodic,derrida1993exact,Uchiyama_2004,Bryc2017,wang2023askey} for selection of works on its structure, in particular on the matrix product ansatz and its connections to the Askey--Wilson orthogonal polynomials. The weakly asymmetric scaling, i.e, where the bias parameter $q$ tends to $1$, takes the ASEP to the KPZ equation,   see Bertini and Giacomin \cite{bertini1997stochastic} for the full-space setting. Corwin and Shen in \cite{corwin2018open}, as well as Parekh in \cite{parekh2017kpz} for general boundary parameters, show that the suitably weakly rescaled open ASEP converges to the open KPZ equation. As in several previous works, we study the open KPZ stationary measure by taking the weakly asymmetric scaling limit from the stationary measure of open ASEP.

The starting point for our arguments is the two-layer representation of the open ASEP. The two-layer Gibbs measures for the stationary measures of integrable models with two-sided boundaries were introduced in \cite{barraquand2024stationary}, later adapted in \cite{bryc2024two} to the open TASEP, and in \cite{bryc2024stationary} for the open ASEP with Liggett's condition. The latter work introduces the family of polynomial correction terms, which we call Bryc polynomials, indexed by integers compositions. Our previous work \cite{SchmidYangShock} characterized these polynomials from a probabilistic perspective using  inhomogeneous birth-death chains; see also \cite{SchmidYangCoexistence} for further applications of this probabilistic characterization on the coexistence line of the open ASEP, and \cite{SchmidYangFan} for complementary results for the open ASEP in the fan region.

\subsection{Strategy for the proof}

Our arguments rely on the two-layer representation of the open ASEP from \cite{bryc2024stationary} together with a probabilistic representation of the associated Bryc polynomials from \cite{SchmidYangShock}. Heuristically, the Bryc polynomials $W_{\bs}$, indexed by integer compositions $\bs$, capture the Radon-Nikodym with respect to the stationary distribution of the totally asymmetric simple exclusion process with open boundaries (open TASEP) where $q=\gamma=\delta=0$. We analyze the Bryc polynomials under a weakly asymmetric scaling, i.e, we evaluate the polynomials $W_{\bs}(z)$ around $z=1$, for typical compositions $\bs$. More precisely, we study the characterization from \cite{SchmidYangShock} under weakly asymmetric scaling, i.e., every Bryc polynomial $W_{\bs}$ for a composition of $N$ is expressed by the moment generating function of an associated time-inhomogeneous birth-death chain $X^{\mathsf S(\bs)}$ evaluated after $N$ steps. Instead of studying $X^{\mathsf S(\bs)}=(X^{\mathsf S(\bs)}_i)_{i \in \{1,\dots,N\}}$ directly, we consider its exponential transformation $(q^{-X^{\mathsf S(\bs)}_i})_{i \in \{1,\dots,N\}}$, which takes a simpler form, and for which we control the increments and moments using martingale arguments. As a consequence, we establish that typical Bryc polynomials allow, after a suitable weakly asymmetric rescaling, for a Taylor expansion around $1$, and we determine the leading coefficient explicitly. Taking now a suitable weakly asymmetric scaling limit of the stationary distribution of the open ASEP as in \cite{corwin2018open,parekh2022ergodicity} to obtain a stationary measure of the open KPZ equation, this allows us  to identify the Radon-Nikodym derivative from Theorem~\ref{thm:main open KPZ stat}.

\subsection{Outline of the paper}

In Section~\ref{sec:TwoLayer}, we collect characterizations of the open ASEP stationary distribution, in particular the two-layer representation of the stationary distribution of the open ASEP via Bryc polynomials, and the probabilistic  characterization of Bryc polynomials via inhomogeneous birth-death chains. In Section~\ref{sec:WeaklyLimit}, we study the inhomogeneous birth-death chain from the probabilistic characterization of the Bryc polynomial, and its exponential transformation, and we establish a limit theorem for typical Bryc polynomials under a weakly asymmetric scaling. In Section~\ref{sec:MainProof}, we conclude the proof of Theorem~\ref{thm:open KPZ equation scaling} on the stationary distribution of the open ASEP under a weakly scaling, and hereby also the proof of Theorem~\ref{thm:main open KPZ stat}.

\subsection*{Acknowledgments} 
D.S. was partially supported by the Packard Foundation via Amol Aggarwal's Packard Fellowships for Science and Engineering.
Z.Y. was partially supported by the Miller Institute for Basic Research in Science at the University of California, Berkeley, Ivan Corwin's NSF grants DMS-1811143, DMS-2246576, and the Fernholz Foundation's ``Summer Minerva Fellows'' program. Both authors were supported by Ivan Corwin's Simons Foundation Grant 929852.

\subsection*{AI declaration} 

This project was initiated in fall 2024, and the proofs were written down by the end of July 2026. The authors used AI to assist with editing the manuscript, performing some computations, and checking arguments. All proof ideas in this article were developed by the authors. 
%
\section{Characterization of the stationary measure of the open ASEP}  \label{sec:TwoLayer}

\subsection{Basic notations}
In this section, we review the characterization of the stationary measure of the open ASEP developed in \cite{bryc2024stationary} and our companion paper \cite{SchmidYangShock}. 

A \emph{lazy random walk path} is $\bg=(\gamma_0,\dots,\gamma_N)$ such that $\gamma_i\in\mathbb{Z}$ for all $i\in\llbracket0,N\rrbracket$ and $\gamma_i-\gamma_{i-1}\in\{-1,0,1\}$ for all $i\in\llbracket1,N\rrbracket$. We denote by $\lrw_N$ the space of lazy random walk paths $\bg$ with length $N$ started at $\gamma_0=0$. For each $\bg\in\lrw_N$, we write
\begin{equation*}
    \min(\bg) = \min\left\{\gamma_i \colon i \in \llbracket 0,N \rrbracket\right\} \quad \text{ and } \quad \max(\bg) = \max\left\{\gamma_i \colon i \in \llbracket 0,N \rrbracket\right\} . 
\end{equation*}
A sequence $\bs=(\sigma_0,\dots,\sigma_r)$, where $r\in\mathbb{Z}_{\geq0}$, is called a composition of $N+1$ if $\sigma_i\in\mathbb{N}$ for $i\in\llbracket0,r\rrbracket$ and $\sigma_0+\dots+\sigma_r=N+1$.  The order in which the elements of $\bs$ appear is significant, which distinguishes it from the notion of a partition.  
For a path $\bg\in\lrw_N$, we define the composition associated to $\bg$, denoted by $\bs(\bg)$, to be the composition $\bs(\bg)=(\sigma_0,\dots,\sigma_r)$ of $N+1$ given by $r=\max(\bg)-\min(\bg)\in\mathbb{Z}_{\geq0}$ and
\begin{equation} \label{eq:Signature}
\sigma_i=\#\{j:\gamma_j=\min(\bg)+i\}\in\mathbb{N}\quad\mbox{ for } i\in\llbracket0,r\rrbracket.
\end{equation}
We note that $\bs(\bg)$ is indeed a composition of $N+1$: since $\bg$ is integer valued and has increments $\{-1,0,1\}$, every integer value in $\llbracket\min(\bg), \max(\bg)\rrbracket$ is visited at least once, and hence $\sigma_i\in\mathbb{N}$ for $i\in\llbracket0,r\rrbracket$; since the level sets $\{j:\gamma_j=\min(\bg)+i\}$ in \eqref{eq:Signature} partition the time indices  $\llbracket0,N\rrbracket$, we have $\sigma_0+\dots+\sigma_r=N+1$.

\subsection{The Bryc polynomial} \label{sec:BrycPoly}

For $q\in[0,1)$ and every composition $\bs=(\sigma_0,\dots,\sigma_r)$ of $N+1$, where $N\in\mathbb{Z}_{\geq0}$, we consider the Bryc polynomial $z \mapsto W_{\bs}(z)$. It was first introduced in \cite{bryc2024stationary} to describe the stationary distribution of the open ASEP. While the construction in \cite{bryc2024stationary} was achieved using Jackson derivative (also called $q$-differentiation operators), we will in the following rely on a probabilistic characterization from \cite{SchmidYangShock} via an inhomogeneous birth-death process.

\begin{definition}\label{def:characterization measure}
Let $q\in[0,1)$, $N\in\mathbb{Z}_{\geq0}$, and $\boldsymbol{\mathsf{S}}=(\mathsf{S}(0),\dots,\mathsf{S}(N))\in\{0,1\}^{N+1}$ be a $0$-$1$ sequence. We define a time-inhomogeneous birth-death Markov chain $X^{\boldsymbol{\mathsf{S}}}=\{X^{\boldsymbol{\mathsf{S}}}(i)\}_{i=0}^{N+1}$ indexed by this $0$-$1$ sequence $\boldsymbol{\mathsf{S}}$ through $X^{\boldsymbol{\mathsf{S}}}(0)=X^{\boldsymbol{\mathsf{S}}}(1)=0$ and that for each $i\in\llbracket1,N\rrbracket$ and $j\in\llbracket0,i-1\rrbracket$,
\begin{enumerate}
    \item [(\romannum{1})] if $\mathsf{S}(i)=1$, we let $X^{\boldsymbol{\mathsf{S}}}(i+1)-X^{\boldsymbol{\mathsf{S}}}(i)\in\{-1,0\}$ with 
    \begin{equation}\label{eq:etget1}
    \mathbb{P}\left(X^{\boldsymbol{\mathsf{S}}}(i+1)=j-1 \Big{|}X^{\boldsymbol{\mathsf{S}}}(i)=j\right)=\frac{1-q^{j}}{1-q^{i+1}},
    \end{equation}
    \item [(\romannum{2})] if $\mathsf{S}(i)=0$, we let $X^{\boldsymbol{\mathsf{S}}}(i+1)-X^{\boldsymbol{\mathsf{S}}}(i)\in\{0,1\}$ with 
    \begin{equation}\label{eq:etget2}
    \mathbb{P}\left(X^{\boldsymbol{\mathsf{S}}}(i+1)=j+1 \Big{|}X^{\boldsymbol{\mathsf{S}}}(i)=j\right)=\frac{q^{j+1}-q^{i+1}}{1-q^{i+1}}.
    \end{equation}
\end{enumerate} 
We mention that the right sides of \eqref{eq:etget1} and \eqref{eq:etget2} lie in $[0,1]$ since $j\in\llbracket0,i-1\rrbracket$ and $q\in[0,1)$. 
Clearly, by an induction argument, we have $X^{\boldsymbol{\mathsf{S}}}(i)\in\llbracket0,i-1\rrbracket$ for all $i\in\llbracket1,N+1\rrbracket$. When $q=0$, we use the convention $0^0=1$.
\end{definition}

\begin{definition}\label{def:associated 01 seq}
    Let $N\in\mathbb{Z}_{\geq0}$ and $\bs=(\sigma_0,\dots,\sigma_r)$, $r\in\mathbb{Z}_{\geq0}$ be a composition of $N+1$. 
We define the $0$-$1$ sequence associated to composition $\bs$ to be $\boldsymbol{\mathsf{S}}^{\bs}=(\mathsf{S}^{\bs}(0),\dots,\mathsf{S}^{\bs}(N))\in\{0,1\}^{N+1}$ such that,  for  $i\in\llbracket0,N\rrbracket$,
\begin{equation}\label{eq:01seq}
  \mathsf{S}^{\bs}(i)=\begin{cases}
        1 &\quad\mbox{ if }i\in\{\sigma_r,\sigma_r+\sigma_{r-1},\dots,\sigma_r+\dots+\sigma_1\},\\  
        0 &\quad\mbox{ otherwise}.
    \end{cases} 
\end{equation}
\end{definition}

\begin{definition}\label{thm:prob interpretation of Bryc polynomials}
Let $q\in[0,1)$, $N\in\mathbb{Z}_{\geq0}$ and $\bs=(\sigma_0,\dots,\sigma_r)$, $r\in\mathbb{Z}_{\geq0}$ be a composition of $N+1$. We define
    \begin{equation}\label{eq:rep of Bryc polynomial}
    W_{\bs}(z):=\mathbb{E}\left[ z^{X^{\boldsymbol{\mathsf{S}}^{\bs}}(N+1)}\right]. 
    \end{equation}
to be the Bryc polynomial associated to $\bs$.  Here, 
$\boldsymbol{\mathsf{S}}^{\bs}$ is the $0$-$1$ sequence associated to the composition $\bs$ as in Definition \ref{def:associated 01 seq}, and $X^{\boldsymbol{\mathsf{S}}^{\bs}}=\{X^{\boldsymbol{\mathsf{S}^{\bs}}}(i)\}_{i=0}^{N+1}$ is the birth-death Markov chain indexed by this $0$-$1$ sequence $\boldsymbol{\mathsf{S}}^{\bs}$ as in Definition \ref{def:characterization measure}. 
\end{definition}

\subsection{Characterization of the stationary measure of the open ASEP} 
Let $\a,\b\geq0$. Using the Bryc polynomials, one can give a simple characterization of the stationary distribution of the open ASEP. We define the probability measure $\mathbb{P}^{(q,\a,\b)}$ on $\lrw_N$ by 
\begin{equation}\label{eq:PNab}
\mathbb{P}^{(q,\a,\b)}(\bg)=\frac{1}{Z_N^{(q,\a,\b)}}\frac{2^{H(\bg)}\b^{\gamma_N}}{(\a\b)^{\min(\bg)}} W_{\bs(\bg)}(\a\b)\quad\mbox{ for }\bg\in\lrw_N,
\end{equation}
where  
\begin{equation}\label{eq:normalization}
 Z_N^{(q,\a,\b)}=\sum_{\bg\in\lrw_N} \frac{2^{H(\bg)}\b^{\gamma_N}}{(\a\b)^{\min(\bg)}}W_{\bs(\bg)}(\a\b).
\end{equation}
Here, we write $H(\bg)$ to denote the number of horizontal steps in a lazy random walk path $\bg$,
\[
H(\bg)=\#\{i\in\llbracket1,N\rrbracket, \gamma_i=\gamma_{i-1}\} \quad\mbox{ for }\bg\in\lrw_N.
\]
We mention that $W_{\bs(\bg)}(\a\b)$ is the Bryc polynomial indexed by the composition $\bs(\bg)$ of $N+1$ associated to $\bg$ (as in \eqref{eq:Signature}), evaluated at $z=\a\b$. Since $\min(\bg)\leq\min(0,\gamma_N)$, we know that 
$ 2^{H(\bg)}\b^{\gamma_N} (\a\b)^{-\min(\bg)}  W_{\bs(\bg)}(\a\b)$ (and hence $Z_N^{(q,\a,\b)}$ given by \eqref{eq:normalization})
is a polynomial in $\a$ and $\b$ with non-negative coefficients depending on $q$.   When $\a=0$ or $\b=0$, \eqref{eq:PNab} is well-defined as a polynomial in $\a$ and $\b$, again with the convention that $0^0=1$. We have the following characterization of the stationary distribution of the open ASEP. 

\begin{theorem}[c.f. Theorem 1.1 and (3.2), (3.3) in \cite{bryc2024stationary}, Theorem 1.5 in \cite{SchmidYangShock}]\label{thm:characterization stationary measure ASEP}
Consider the open ASEP on the lattice $\llbracket1,N\rrbracket$ with jump rates $\alpha,\beta,\gamma,\delta$ and $q$ satisfying \eqref{eq:conditions open ASEP}.  Assume that Liggett's condition \eqref{eq:Liggett's condition} holds, and recall that $\a=(1-\alpha)/\alpha$ and $\b=(1-\beta)/\beta$.  We denote the stationary measure of this model by $\mu_N^{(q,\a,\b)}$.
Consider the random lazy random walk path $\bg\in\lrw_N$ with law $\mathbb{P}^{(q,\a,\b)}$ defined as in \eqref{eq:PNab}. Suppose $\kappa_i$, $i\in\llbracket1,N\rrbracket$ are i.i.d. Bernoulli random variables with mean $\frac{1}{2}$ that are independent of $\bg$. 
We define random variables $\tau_i\in\{0,1\}$, $i\in\llbracket1,N\rrbracket$ by  
\[
    \tau_i=\begin{cases}
        1 & \mbox{ if } \gamma_i-\gamma_{i-1}=1,\\
        \kappa_i & \mbox{ if } \gamma_i-\gamma_{i-1}=0 ,\\ 
        0 & \mbox{ if } \gamma_i-\gamma_{i-1}=-1.
    \end{cases}
\]
Then $(\tau_1,\dots,\tau_N)\in\{0,1\}^N$ has the law of the stationary measure $\mu_N^{(q,\a,\b)}$ of the open ASEP.
\end{theorem}

\section{The weakly asymmetric limit of Bryc polynomials}\label{sec:WeaklyLimit}

In the following, our goal is to derive a first order approximation of the Bryc polynomials under a weakly asymmetric scaling for typical lazy random walk paths. This is summarized in the next proposition.

\begin{proposition}
\label{prop:weak asym Bryc}
Fix \(L>0\) and \(\lambda\in\mathbb R\). Let
\(q_n=e^{-2/\sqrt n}\) for \(n\in\mathbb N\). For
\(\bg\in\lrw_{\lfloor nL\rfloor}\), the Bryc polynomial
\(W_{\bs(\bg)}\) is evaluated with parameter \(q=q_n\). For \(R\ge0\),
denote
\[
\lrw(n,L,R)=
\left\{
\bg\in\lrw_{\lfloor nL\rfloor}:
\left| \gamma_{\lfloor nx\rfloor}\right|
\le Rx\sqrt n\mbox{ for every }x\in[0,L]
\right\}.
\]
Then, for every \(R\ge0\),
\[
\lim_{n\to\infty}\sup_{\bg\in\lrw(n,L,R)}
\left|
\left(2\sqrt n\right)^\lambda
q_n^{-\lambda\min(\bg)}
W_{\bs(\bg)}\left(q_n^\lambda\right)
-
\left(  \frac1n\sum_{j=0}^{\lfloor nL\rfloor} q_n^{\gamma_j}\right)^{-\lambda}
\right|=0 .
\]
Moreover, there exists a constant \(C<\infty\), depending only on \(L\)
and \(\lambda\), such that for every \(R\ge0\), every \(n\in\mathbb N\),
and every \(\bg\in\lrw(n,L,R)\),
\[
\left(2\sqrt n\right)^\lambda
q_n^{-\lambda\min(\bg)}
W_{\bs(\bg)}\left(q_n^\lambda\right)
\le
Ce^{CR}.
\]
\end{proposition}

As a first step, we study a suitable process $(Y_i)_{i\in \llbracket0,N+1\rrbracket}$ which we obtain from the birth-death process $(X_i)_{i\in \llbracket0,N+1\rrbracket}$ by exponentiation. Here, and in the following, we  write
\begin{equation}
    [n]_q := (1-q^n)/(1-q)
\end{equation} for the q-analogue of a natural number $n \in \mathbb N$.  

\begin{lemma}\label{lem:Y chain}
     Assume $q\in(0,1)$ and $N\in\mathbb{N}$. Let $\bs=(\sigma_0,\dots,\sigma_r)$ be a composition of $N+1$ and $\boldsymbol{\mathsf S}^{\bs}\in\{0,1\}^{N+1}$ be the associated $0$-$1$ sequence as in Definition \ref{def:associated 01 seq}. Let $X=X^{\boldsymbol{\mathsf S}^{\bs}}$ be the birth-death chain indexed by $\boldsymbol{\mathsf S}^{\bs}$ as in Definition \ref{def:characterization measure}, and write $\mathsf S=\mathsf S^{\bs}$. Let $Y_i=[i]_q\,q^{-X(i)} $ for $i\in\llbracket0,N+1\rrbracket$. Then $Y_1=1$,
\begin{equation}\label{eq:jtyukmtyufk}
\mathbb E[Y_{i+1}\mid X(i)]=\left(1+(q-1){\bf1}_{\{\mathsf S(i)=1\}}\right)Y_i+1,\quad\quad\mathbb{E}\left[Y_{i+1}\right]=\left(1+(q-1){\bf1}_{\{\mathsf S(i)=1\}}\right) \mathbb{E}\left[Y_i\right]+1 ,
\end{equation}
for $i\in\llbracket1,N\rrbracket$.
Moreover, $\mathbb{E}[Y_i]\ge [i]_q$ for $i\in\llbracket1,N+1\rrbracket$, and
$\mathbb{E}\left[Y_{N+1}\right] 
=
\sum_{\ell=0}^r\sigma_\ell q^\ell$.
\end{lemma} 

\begin{proof} In this proof, we denote $X=X^{\boldsymbol{\mathsf S}^{\bs}}$ and $\mathsf S=\mathsf S^{\bs}$. Note $Y_1=1$ since $X(1)=0$. Fix $i\in\llbracket1,N\rrbracket$ and $j\in\llbracket0,i-1\rrbracket$. We compute the one-step transition-kernel expectation from the state \(j\). If $\mathsf{S}(i)=0$, then
\[
\mathbb{P}\left(X(i+1)=j+1 \mid X(i)=j\right)=\frac{q^{j+1}-q^{i+1}}{1-q^{i+1}},\qquad
\mathbb{P}\left(X(i+1)=j  \mid X(i)=j\right)=\frac{1-q^{j+1} }{1-q^{i+1}},
\]
and hence 
$\mathbb E[Y_{i+1}\mid X(i)=j]=[i]_q q^{-j}+1$.
If $\mathsf{S}(i)=1$, then 
\[
\mathbb{P}\left(X(i+1)=j-1 \mid X(i)=j\right)=\frac{1-q^{j}}{1-q^{i+1}},\qquad
\mathbb{P}\left(X(i+1)=j \mid X(i)=j\right)=\frac{ q^{j}-q^{i+1}}{1-q^{i+1}},
\]
and hence
$\mathbb E[Y_{i+1}\mid X(i)=j]=q[i]_q q^{-j}+1 $.
Since \(Y_i=[i]_q q^{-X(i)}\) and \(j\) was arbitrary, these two identities imply the first equality in \eqref{eq:jtyukmtyufk}; the second equality follows from taking expectations.
Note that \eqref{eq:jtyukmtyufk} implies $\mathbb{E} [Y_{i+1} ]\geq q \mathbb{E} [Y_i ]+1$, which, by induction, implies $\mathbb{E}[Y_i]\ge [i]_q$ for $i\in\llbracket1,N+1\rrbracket$, since we have  $[i+1]_q=q[i]_q+1$.
For the last statement, for \(\ell\in\llbracket0,r\rrbracket\), set \(p_\ell=\sigma_r+\cdots+\sigma_\ell\). By Definition \ref{def:associated 01 seq}, these \(p_\ell\)'s except for $p_0=N+1$ are exactly the indices \(i\in\llbracket1,N\rrbracket\) for which \(\mathsf S(i)=1\). If \(r=0\), then the above recursion gives \(\mathbb E[Y_{N+1}] =N+1=\sigma_0\). Assume \(r\ge1\). Starting from \(\mathbb E[Y_1]=1\), the recurrence gives \(\mathbb E[Y_{p_r}] =\sigma_r\). If
\[
\mathbb E[Y_{p_{\ell}}]=\sum_{a=\ell}^r\sigma_a q^{a-\ell}
\]
for some \(\ell\in\llbracket1,r\rrbracket\), then the update at \(p_\ell\) and the \(\sigma_{\ell-1}-1\) following updates with \(\mathsf S(i)=0\) give
\[
\mathbb E[Y_{p_{\ell-1}}]=q\,\mathbb E[Y_{p_{\ell}}]+\sigma_{\ell-1}
=\sum_{a=\ell-1}^r\sigma_a q^{a-\ell+1}.
\]
Therefore,
\[
\mathbb E[Y_{N+1}]= \mathbb E[Y_{p_0}]=\sum_{\ell=0}^r\sigma_\ell q^\ell .
\]
We conclude the proof.
\end{proof}

Next, we study the law of the process $(Y_i)_{i\in \llbracket0,N+1\rrbracket}$ at time $i+1$ when conditioning on the trajectory up to time $i$, compare it to the unconditioned process, and bound the moments of the differences.

\begin{lemma}\label{lem:center Y chain}
Assume the same notation as Lemma \ref{lem:Y chain}. Assume $q\in[1/2,1)$. Let $\mathcal{F}_i=\sigma(Y_1,\dots,Y_i)$ for $i\in\llbracket1,N+1\rrbracket$ and $\Delta_{i+1}=Y_{i+1}-\mathbb E[Y_{i+1}\mid\mathcal F_i]$.
 Then for every \(p\ge2\) and $i\in\llbracket1,N\rrbracket$, 
\begin{equation*}
    \mathbb E\left[\left|\Delta_{i+1}\right|^p\mid\mathcal F_i\right]
\le
2^p(1-q)^{p-1}(Y_i+1)^{p-1} .
\end{equation*} 
\end{lemma}
\begin{proof}
For a two-point random variable \(Z\) with values separated by \(d\) and probabilities \(\theta\) and \(1-\theta\),
\begin{equation}\label{eq:two point RW}
\mathbb E|Z-\mathbb EZ|^p
=
\theta(1-\theta)\left((1-\theta)^{p-1}+\theta^{p-1}\right)d^p
\le
\theta(1-\theta)d^p .
\end{equation}
In this proof, we write \(X=X^{\boldsymbol{\mathsf S}^{\bs}}\) and fix $i\in\llbracket1,N\rrbracket$. Since \(q\in(0,1)\), the map \(x\mapsto [k]_q q^{-x}\) is injective for each \(k\ge1\), and hence \(\mathcal F_i=\sigma(X(1),\dots,X(i))\). By the Markov property, \(\mathbb E[Y_{i+1}\mid\mathcal F_i]=\mathbb E[Y_{i+1}\mid X(i)]\). Therefore it suffices to prove the desired bound on each event \(\{X(i)=j\}\) with positive probability, where \(j\in\llbracket0,i-1\rrbracket\). 
If \(\mathsf S^{\bs}(i)=0\), then \(Y_{i+1}\) is supported on \([i+1]_q q^{-j}\) and \([i+1]_q q^{-j-1}\).  The distance is
$d=(q^{-1}-1)[i+1]_q q^{-j}\le 2(1-q)([i]_q q^{-j}+1)$,
since \(q\ge1/2\) and \([i+1]_q q^{-j}=[i]_q q^{-j}+q^{i-j}\le [i]_q q^{-j}+1\).  The two probabilities are \(\theta\) and \(1-\theta\), where $\theta=(q^{j+1}-q^{i+1})/(1-q^{i+1})$. Using the fact that \((1-a)(1-b)\le 1-ab\) for \(a,b\in[0,1]\),  
\[
\theta(1-\theta)d
=
\left(1-q^{i-j}\right)\left(1-q^{j+1}\right)\left(1-q^{i+1} \right)^{-1}
\le 1.
\]
If \(\mathsf S^{\bs}(i)=1\), then \(Y_{i+1}\) is supported on $[i+1]_q q^{-j}$ and $[i+1]_q q^{-j+1}$. 
The distance is
$d=(1-q)[i+1]_q q^{-j}\le (1-q)([i]_q q^{-j}+1)$.
The two probabilities are \(\theta\) and \(1-\theta\), where $\theta=(1-q^j)/(1-q^{i+1})$. Again, we have
\[
\theta(1-\theta)d
=
 \left(1-q^j\right)\left(1-q^{i+1-j}\right)\left(1-q^{i+1}\right)^{-1}
\le1.
\] 
Plugging in the upper bounds for \(d\) and the displayed bounds on \(\theta(1-\theta)d\) in both cases into \eqref{eq:two point RW},
\[ 
\mathbb E\left[\left|\Delta_{i+1}\right|^p\mid X(i)=j\right]
 \le
\theta(1-\theta)d^p \le
2^p (1-q)^{p-1}\left([i]_q q^{-j}+1\right)^{p-1}. 
\]
Since we have $X(i)=j$ if and only if $Y_i=[i]_q q^{-j}$, and the events \(\{X(i)=j\}\) partition the state space up, we conclude the proof.
\end{proof}

Next, we use Lemma~\ref{lem:Y chain} to control for all $i\in \llbracket 1,N\rrbracket$ the moments of $Y_{i+1}$ in terms of the respective moments of $Y_i$.

\begin{lemma}\label{lem:martingale V}
Assume the same notation as Lemma \ref{lem:Y chain}  and  $q\in[1/2,1)$. Then for every \(p\ge2\) and $i\in\llbracket1,N\rrbracket$,
\[
\mathbb E\left[V_{i+1}^p\right] 
\le
\left(1+ p(p-1)\left(2^{4p-4}+2^{2p}\right) (1-q) \cdot m_{i+1}^{-1}\right)\cdot \mathbb E\left[V_i^p\right],
\]
where we set $m_k=\mathbb E[Y_k]$ and $V_k=Y_k/m_k$.
\end{lemma} 
\begin{proof}
We will use the following deterministic inequality: for any \(x>0\), \(x+u\ge0\), and \(p\ge2\),
\begin{equation}\label{eq:elementarytaylor}
(x+u)^p
\le
x^p+p x^{p-1}u+p(p-1)2^{p-2}\left(x^{p-2}u^2+|u|^p\right). 
\end{equation}
Indeed, Taylor's formula gives
\[
(x+u)^p=x^p+p x^{p-1}u
+p(p-1)u^2\int_0^1(1-s)(x+su)^{p-2}\,ds,
\]
and, since \(x+su\ge0\) for \(s\in[0,1]\) for the first inequality, as well as  $x+|u|\leq 2 \max(x,|u|)$ for the second inequality, we have
\[
u^2(x+su)^{p-2}
\le u^2(x+|u|)^{p-2}
\le 2^{p-2}\left(x^{p-2}u^2+|u|^p\right),
\]
and we conclude \eqref{eq:elementarytaylor}.
In the rest of the proof, we denote $\mathcal F_i=\sigma(Y_1,\dots,Y_i)$,
    \[
     a_i= 1+(q-1){\bf1}_{\{\mathsf S^{\bs }(i)=1\}}\quad\mbox{and}\quad b_i= \frac{a_i m_i}{m_{i+1}}.
    \]
Since \(q\in(0,1)\), for each \(1\le k\le i\) the variable
\(Y_k=[k]_q q^{-X^{\boldsymbol{\mathsf S}^{\bs}}(k)}\) determines
\(X^{\boldsymbol{\mathsf S}^{\bs}}(k)\). Hence
$\mathcal F_i
=
\sigma\left(X^{\boldsymbol{\mathsf S}^{\bs}}(1),\dots,
X^{\boldsymbol{\mathsf S}^{\bs}}(i)\right)$.
By the Markov property of \(X^{\boldsymbol{\mathsf S}^{\bs}}\) and Lemma
\ref{lem:Y chain},
$\mathbb E[Y_{i+1}\mid\mathcal F_i]=a_iY_i+1 $.
Taking expectations gives \(m_{i+1}=a_i m_i+1\), hence \(b_i\in(0,1)\).
Set
\begin{equation}\label{eq:xu}
\Delta_{i+1}=Y_{i+1}-\mathbb E[Y_{i+1}\mid\mathcal F_i],
\qquad
x=b_iV_i+(1-b_i),
\qquad
u=\frac{\Delta_{i+1}}{m_{i+1}}.
\end{equation} 
Note that \(x=(a_iY_i+1)/m_{i+1}>0\) and \(x+u=V_{i+1}\ge0\). Applying
\eqref{eq:elementarytaylor} with these choices, taking conditional
expectations, and using \(\mathbb E[u\mid\mathcal F_i]=0\), we obtain 
\begin{equation}\label{eq:brtnwyn}
\mathbb E[V_{i+1}^p\mid\mathcal F_i]
\le
x^p+p(p-1)2^{p-2}\left( x^{p-2}\mathbb E[u^2\mid\mathcal F_i]
+ \mathbb E[|u|^p\mid\mathcal F_i]\right).
\end{equation} 
We next bound the terms on the right side above. Using Lemma \ref{lem:center Y chain} with exponents \(2\) and \(p\), we get
\[
\mathbb E[u^2\mid\mathcal F_i]
\le
4\frac{(1-q)(Y_i+1)}{m_{i+1}^2},
\qquad
\mathbb E[|u|^p\mid\mathcal F_i]
\le
2^p\frac{(1-q)^{p-1}(Y_i+1)^{p-1}}{m_{i+1}^p}.
\]
Combining with \(x\le V_i+1\) since \(b_i\in(0,1)\) and
\(V_i\ge0\), with  \(Y_i+1=m_iV_i+1\le
2m_{i+1}(V_i+1)\) since \(m_{i+1}=a_i m_i+1\ge q m_i\),
\(m_{i+1}\ge1\), and \(q\ge1/2\), and with
\((1-q)^{p-1}\le 1-q\), the above display implies 
\[
x^{p-2}\mathbb E[u^2\mid\mathcal F_i]
+\mathbb E[|u|^p\mid\mathcal F_i]
\le
\left(8+2^{2p-1}\right)\frac{1-q}{m_{i+1}}(V_i+1)^{p-1}.
\]
Combining with \eqref{eq:brtnwyn},
\begin{equation}\label{eq:73563}
\mathbb E[V_{i+1}^p\mid\mathcal F_i]
\le
x^p+p(p-1)2^{p-2}\left(8+2^{2p-1}\right)\frac{1-q}{m_{i+1}}(V_i+1)^{p-1}.
\end{equation}
Convexity gives \(x^p\le b_iV_i^p+(1-b_i)\).    Since \(\mathbb E[V_i]=1\), Jensen's inequality gives \(\mathbb E[V_i^p]\ge1\), and hence
\[
\mathbb E[x^p]\le b_i\mathbb E[V_i^p]+(1-b_i)\le \mathbb E[V_i^p].
\]
Moreover,
\[
\mathbb E[(V_i+1)^{p-1}]
\le 2^{p-2}\left(\mathbb E[V_i^{p-1}]+1\right)
\le 2^{p-2}\left(\mathbb E[V_i^p]^{(p-1)/p}+1\right)
\le 2^{p-1}\mathbb E[V_i^p].
\]
Taking expectations in \eqref{eq:73563} and combining with the above display, we conclude the proof.
\end{proof}

The preceding steps yield now for all $p\geq 2$ a uniform bound on the $p^{\text{th}}$ moments of $Y_i$, uniformly in $i\in \llbracket 1,N\rrbracket$ and the chosen composition $\bs$.

\begin{lemma}\label{lem:maximum V}
Fix \(L>0\) and \(p\ge2\). Assume the same notation as Lemma \ref{lem:Y chain}, with $q=e^{-2/\sqrt n}$ and $N=\lfloor nL\rfloor$ for $n\in\mathbb N$ large enough so that $q\geq1/2$ and $N\geq1$. 
Denote $m_k=\mathbb E[Y_k]$ and $V_k=Y_k/m_k$.
Then there exists a finite constant \(C_{L,p}\), depending only on \(L\) and \(p\), such that, for every such \(n\),
\[
\sup_{\bs}\sup_{ i\in\llbracket1, N+1\rrbracket}\mathbb E\left[V_i^p\right]\le C_{L,p},
\]
where the first supremum is over all compositions \(\bs\) of \(N+1\).
\end{lemma}
\begin{proof}
We have for some constant $C_{L}>0$
\begin{equation}\label{eq:ndtmeum}
        \begin{split}
           & \sum_{i=1}^N\frac{1-q}{m_{i+1}}
\le
\sum_{k=2}^{N+1}\frac{(1-q)^2}{1-q^k}
\leq  
\left(1-e^{-1}\right)^{-1}\sum_{k=2}^{N+1}\frac{(1-q)^2}{\min(1,k(1-q))}  \\
&\le
\left(1-e^{-1}\right)^{-1}(1-q)\sum_{  k\in\llbracket2, \min (N+1,\lfloor (1-q)^{-1}\rfloor)\rrbracket}\frac1k
+
\left(1-e^{-1}\right)^{-1}(1-q)^2 N\le C_L ,
        \end{split}
    \end{equation}
where the first step follows from \(m_k\ge [k]_q=(1-q^k)/(1-q)\) by Lemma \ref{lem:Y chain}, the second step follows from 
\[
 1-q^k\ge1-e^{-k(1-q)}\geq \left(1-e^{-1}\right)\min(1,k(1-q)),
\]
the third step splits the summation over $k$ into two parts, and the fourth step uses 
\[
(1-q)\sum_{k=2}^{\lfloor(1-q)^{-1}\rfloor}\frac1k
\le
\sup_{x\in(0,1/2]}x\left(1+\log\left(x^{-1}\right)\right)<\infty,
\qquad
(1-q)^2N\le 4L ,
\]
where we have used \(1-q\le 2/\sqrt n\) together with a change of variables $x=1-q$. For every \(i\in\llbracket1,N+1\rrbracket\),
\[
\mathbb E\left[V_i^p\right]
\le
\prod_{j=1}^{i-1}\left(1+A_p\frac{1-q}{m_{j+1}}\right)
\le
\exp\left(A_p\sum_{j=1}^{N}\frac{1-q}{m_{j+1}}\right)
\le e^{A_pC_L},
\]
where we write \(A_p=p(p-1)(2^{4p-4}+2^{2p})\). Here, the first step uses \(V_1=1\) combined with Lemma \ref{lem:martingale V}; the second step uses $1+x\leq e^x$ for $x\geq0$; and the third step uses \eqref{eq:ndtmeum}. Combining the constants $A_p$ and $C_L$, we conclude the proof.
\end{proof}

Next, we aim to show Lemma~\ref{lem:concentration M} below, stating that the renormalized birth-death chain endpoint $V_{N+1}$ is concentrated around $1$. To this end, we will use Rosenthal's inequality on martingale differences, which we recall as the following lemma.

\begin{lemma}\label{lem:martingale cite}
    Let \(m\in\mathbb N\), let \((\xi_i)_{i=1}^m\) be real-valued martingale differences with respect to a filtration \((\mathcal G_i)_{i=0}^m\), and assume that \(\mathbb E|\xi_i|^p<\infty\) for every \(i\in\llbracket 1,m\rrbracket\), for some \(2\le p<\infty\). Then there exists a finite constant \(C_p<\infty\), depending only on \(p\), such that
  \[
\mathbb E\left|\sum_{i=1}^m \xi_i\right|^p
\le
C_p\,
\mathbb E\left(\sum_{i=1}^m
\mathbb E\left[\xi_i^2\mid\mathcal G_{i-1}\right]\right)^{p/2}
+
C_p\sum_{i=1}^m\mathbb E|\xi_i|^p .
\]
\end{lemma}
\begin{proof}
    This result can be found as \cite[Theorem 2.12]{HallHeyde1980}; see also \cite[(5.72)]{Jiang2010} for a standard reference.
\end{proof}

\begin{lemma}\label{lem:concentration M}
Fix \(L>0\) and \(p\ge2\). For \(n\in\mathbb N\) large enough so that \(q=e^{-2/\sqrt n}\ge1/2\) and \(N=\lfloor nL\rfloor\ge1\), use the notation of Lemma \ref{lem:maximum V}. Then there exists a finite constant \(C_{L,p}\), depending only on \(L\) and \(p\), such that, for every such \(n\) and every composition \(\bs\) of \(N+1\),
    \[
 \mathbb E\left[
\left|
V_{N+1} -1
\right|^p
\right]
\le C_{L,p}(1-q)^{p/2}.
\]
\end{lemma}
\begin{proof}
    Write \(\mathcal F_i=\sigma(Y_1,\dots,Y_i)\), as well as
    \[
a_i=1+(q-1){\bf1}_{\{\mathsf S^{\bs }(i)=1\}},
\qquad
A_i=\prod_{k=i+1}^N a_k,
\qquad
\Delta_{i+1}=Y_{i+1}-\mathbb E[Y_{i+1}\mid\mathcal F_i].
\]
Since \(q\in(0,1)\) and \(Y_k=[k]_q q^{-X^{\boldsymbol{\mathsf S}^{\bs}}(k)}\), we have
\(\mathcal F_i=\sigma(X^{\boldsymbol{\mathsf S}^{\bs}}(1),\dots,
X^{\boldsymbol{\mathsf S}^{\bs}}(i))\). Hence, by the Markov property and
Lemma \ref{lem:Y chain}, $\mathbb E[Y_{i+1}\mid\mathcal F_i]=a_iY_i+1$ and
$m_{i+1}=a_i m_i+1$.
We claim that  
\begin{equation}\label{eq:nrsyth}
Y_{N+1}-m_{N+1}=\sum_{i=1}^N A_i\Delta_{i+1},\qquad
\sum_{i=1}^N A_i\le m_{N+1},
\qquad
A_i m_i\le2m_{N+1} .
\end{equation}
Indeed, we have
\[
Y_{i+1}-m_{i+1}=a_i(Y_i-m_i)+\Delta_{i+1},
\quad\mbox{for } i\in\llbracket1, N\rrbracket.
\]
Since \(Y_1=m_1=1\), unrolling this recursion concludes the first identity in \eqref{eq:nrsyth}.
Iterating \(m_{i+1}=a_i m_i+1\),
\[
m_i=1+\sum_{j=1}^{i-1}\prod_{k=j}^{i-1}a_k.
\]
Since \(a_i\in[q,1]\),
\[
\sum_{i=1}^N A_i-\sum_{i=1}^N a_iA_i
=
\sum_{i=1}^N(1-a_i)\prod_{k=i+1}^N a_k
=
1-\prod_{k=1}^N a_k
\le1.
\]
Combining with \(m_{N+1}=1+\sum_{i=1}^N a_iA_i\), this implies the second claim in \eqref{eq:nrsyth}.
For each \(i\in\llbracket1,N\rrbracket\),
\[
a_iA_i m_i
=
\sum_{j=1}^i\prod_{k=j}^N a_k
\le m_{N+1}.
\]
Since \(a_i\ge q\ge \frac{1}{2}\), we conclude the third claim in \eqref{eq:nrsyth}.
We next claim that 
\begin{equation}\label{eq:momentbrtwrtbrt}
\mathbb E[(Y_i+1)^\rho]\le C_{L,p}m_i^\rho \quad\mbox{for }\rho\in[1,p]\mbox{ and }i\in\llbracket1,N+1\rrbracket.
\end{equation}
Indeed, Lemma \ref{lem:Y chain} gives \(m_i\ge [i]_q\ge1\), and hence
$Y_i+1=m_iV_i+1\le m_i(V_i+1)$.
Since \((V_i+1)^\rho\le C_p(1+V_i^p)\) for \(\rho\in[1,p]\), Lemma
\ref{lem:maximum V} implies \eqref{eq:momentbrtwrtbrt}.

We now apply Lemma \ref{lem:martingale cite} with \(m=N\), with \(\xi_i=A_i\Delta_{i+1}\) for \(i\in\llbracket1,N\rrbracket\), and with \(\mathcal G_i=\mathcal F_{i+1}\) for \(i\in\llbracket0,N\rrbracket\). Indeed, \(\xi_i\) is \(\mathcal G_i\)-measurable, \(\mathbb E[\xi_i\mid\mathcal G_{i-1}]=A_i\mathbb E[\Delta_{i+1}\mid\mathcal F_i]=0\), and \(\xi_i\) is \(p\)-integrable because the chain has finite state space. We obtain
\begin{equation}\label{eq:rbsnyw3}
\begin{split}
 &   \mathbb E|Y_{N+1}-m_{N+1}|^p
\le
C_p\,
\mathbb E\left(
\sum_{i=1}^N
A_i^2\mathbb E[\Delta_{i+1}^2\mid\mathcal F_i]
\right)^{p/2}
+
C_p\sum_{i=1}^N A_i^p\mathbb E|\Delta_{i+1}|^p \\
&\leq C_p(1-q)^{p/2}
\mathbb E\left(\sum_{i=1}^N A_i^2(Y_i+1)\right)^{p/2}
+
C_p\sum_{i=1}^N A_i^p\mathbb E|\Delta_{i+1}|^p,
\end{split}\end{equation}
where the second step follows from Lemma \ref{lem:center Y chain} with \(p=2\).

We next bound the two summands on the right side above. We have
\begin{equation}\label{eq:rbsnyw1}
\begin{aligned}
&\mathbb E\left(\sum_{i=1}^N A_i^2(Y_i+1)\right)^{p/2}
\le
\left(\sum_{i=1}^N A_i\right)^{p/2-1}
\sum_{i=1}^N A_i^{p/2+1}\mathbb E\left[(Y_i+1)^{p/2}\right]  \\
&\le
C_{L,p}m_{N+1}^{p/2-1}
\sum_{i=1}^N A_i(A_i m_i)^{p/2} 
\le
C_{L,p}m_{N+1}^{p/2-1}\sum_{i=1}^N A_im_{N+1}^{p/2}
\le
C_{L,p}m_{N+1}^{p},
\end{aligned}
\end{equation}
where the first step uses Hölder's inequality; the second step uses \eqref{eq:momentbrtwrtbrt} with $\rho=p/2$ and the second claim in \eqref{eq:nrsyth}; the third step uses the third claim in \eqref{eq:nrsyth}; and the last step uses the second claim in \eqref{eq:nrsyth}.
Note
\begin{equation}\label{eq:rbsnyw2}
\begin{aligned}
&\sum_{i=1}^N A_i^p\mathbb E\left|\Delta_{i+1}\right|^p
\leq \sum_{i=1}^N A_i^p2^p(1-q)^{p-1}\mathbb E\left[(Y_i+1)^{p-1}\right]
\le
C_{L,p}(1-q)^{p-1}
\sum_{i=1}^N A_i^p m_i^{p-1} \\  
&=
C_{L,p}(1-q)^{p-1}
\sum_{i=1}^N A_i(A_i m_i)^{p-1}  
\le
C_{L,p}(1-q)^{p-1}m_{N+1}^p ,
\end{aligned}
\end{equation}
where the first step uses Lemma \ref{lem:center Y chain}; the second step uses \eqref{eq:momentbrtwrtbrt} with $\rho=p-1$; the third step changes the form of the equation; and the last step follows from combining the last two claims in \eqref{eq:nrsyth}.

Plugging \eqref{eq:rbsnyw1} and \eqref{eq:rbsnyw2} into \eqref{eq:rbsnyw3},
\[
\begin{aligned}
 \mathbb E|Y_{N+1}-m_{N+1}|^p  
 \le
C_{L,p}(1-q)^{p/2}m_{N+1}^p
+
C_{L,p}(1-q)^{p-1}m_{N+1}^p \le
C_{L,p}(1-q)^{p/2}m_{N+1}^p,
\end{aligned}
\]
where the second step follows from \(p-1\geq p/2\) and \(1-q\le1\).
Since \(V_{N+1}=Y_{N+1}/m_{N+1}\), dividing both sides of the above display by \(m_{N+1}^p\) concludes the proof.
\end{proof}

 As a consequence of Lemma~\ref{lem:concentration M}, we obtain the desired control on all moments of $V_{N+1}$.

\begin{lemma} \label{lem:VN+1 power lambda}
Fix \(L>0\) and \(\lambda\in\mathbb R\). Let  \(q=e^{-2/\sqrt n}\), \(N=\lfloor nL\rfloor\), and use the notation of Lemma \ref{lem:maximum V}. Then 
\[ 
\lim_{n\to\infty}\sup_{\bs}
\left|
\mathbb E\left[
 V_{N+1}^{-\lambda}
\right]-1
\right|
=0 ,
\]
where the supremum is over all compositions \(\bs\) of \(N+1\). 
\end{lemma}
\begin{proof}
    In this proof, we write $V=V_{N+1}$ for simplicity. We take \(n\) large enough that \(q\ge1/2\) and \(N\ge1\). If \(\lambda=0\), then \(\mathbb E[V^{-\lambda}]=1\). We split the rest of the proof into two cases.

    First, we assume $\lambda<0$. Choose \(p>\max\{2,-\lambda\}\). By Lemmas \ref{lem:concentration M} and \ref{lem:maximum V},
\begin{equation}\label{eq:VN+1 power lambda proof1}
\sup_{\bs}\mathbb E|V-1|^p\le C_{L,p}(1-q)^{p/2}
 ,\qquad \sup_{\bs}\mathbb E[V^p]\le C_{L,p}.
\end{equation}
Fix \(0<\eta<1/2\), and let \(A_\eta=\{|V-1|>\eta\}\). We have 
\begin{equation}\label{eq:VN+1 power lambda proof3}
\mathbb P(A_\eta)
\le
\eta^{-p}\mathbb E|V-1|^p
\le \eta^{-p} C_{L,p}(1-q)^{p/2},
\end{equation}
where the first step uses Markov's inequality and the second step uses \eqref{eq:VN+1 power lambda proof1}. 
Therefore,
\begin{equation}\label{eq:VN+1 power lambda proof2}
\mathbb E\left[V^{-\lambda}\mathbf 1_{A_\eta}\right]
\le
\left(\mathbb E[V^p]\right)^{-\lambda/p}
\mathbb P(A_\eta)^{1+\lambda/p}
\le
C_{L,p}^{-\lambda/p}\mathbb P(A_\eta)^{1+\lambda/p}\leq C_{L,p}^{-\lambda/p}\left( \eta^{-p} C_{L,p}(1-q)^{p/2}\right)^{1+\lambda/p},
\end{equation}
where the first step uses Hölder's inequality and \(-\lambda<p\) and the second step uses \eqref{eq:VN+1 power lambda proof1}.  Hence,
\begin{equation}\label{eq:VN+1 power lambda proof4}
    \begin{split}
        &\sup_{\bs}\mathbb E\left|V^{-\lambda}-1\right|\leq  \sup_{x\in[1-\eta,1+\eta]} \left|x^{-\lambda}-1\right| 
        +\sup_{\bs}\mathbb E\left[V^{-\lambda}\mathbf 1_{A_\eta}\right]+\sup_{\bs}\mathbb{P}\left(A_\eta \right)\\
        &\leq\sup_{x\in[1-\eta,1+\eta]} \left|x^{-\lambda}-1\right| + \eta^{-p} C_{L,p}(1-q)^{p/2}+C_{L,p}^{-\lambda/p}\left( \eta^{-p} C_{L,p}(1-q)^{p/2}\right)^{1+\lambda/p},
    \end{split}
\end{equation}
where the first step splits the expectation into $A_{\eta}$ and $A_{\eta}^{\complement}$ parts and uses the fact that
\[
\left|V^{-\lambda}-1\right|\le  
\sup_{x\in[1-\eta,1+\eta]} \left|x^{-\lambda}-1\right|\quad\mbox{on }A_\eta^{\complement};
\]
the second step uses \eqref{eq:VN+1 power lambda proof3}, \eqref{eq:VN+1 power lambda proof2}.  
Taking first $n\to\infty$ and then $\eta\to0+$ in \eqref{eq:VN+1 power lambda proof4} concludes the proof for $\lambda<0$.

Next, we assume \(\lambda>0\). Choose \(p>\max\{2,2\lambda\}\). On $\{V\ge1/2\}$, we have
\begin{equation}\label{eq:VN+1 power lambda proof5}
\sup_{\bs}
\mathbb E\left[
\left|V^{-\lambda}-1\right|\mathbf 1_{\{V\ge1/2\}}
\right]
\leq \lambda 2^{\lambda+1}\sup_{\bs}\mathbb{E}\left[ |V-1|\right]
\le
C_{\lambda,L}(1-q)^{1/2},
\end{equation}
where the first step uses the fact that the derivative in $x$ satisfies $|\partial_x(x^{-\lambda})|\leq \lambda 2^{\lambda+1}$ for all  $x\geq1/2$, while the second step follows from Cauchy--Schwarz and Lemma \ref{lem:concentration M} with \(p=2\). To bound the expectation on $\{V<1/2\}$, we note that
\[
 m_{N+1}/[N+1]_q 
\leq (N+1)/[N+1]_q 
\le
2 (1-q)(N+1)
\le
8(L+1) (1-q)^{-1}
\]
for sufficient large $n$. Here, the first step uses \(m_{N+1}\leq N+1\), which follows from Lemma \ref{lem:Y chain} since \(m_{N+1}=\sum_{\ell=0}^r\sigma_\ell q^\ell\) and \(\sum_{\ell=0}^r\sigma_\ell=N+1\); the second step uses the fact that for large enough $n$, $1-q^{N+1}\geq1/2$; and the third step uses  that $(1-q)^2(N+1)\le 4(L+1)$ since $1-q\le 2/\sqrt n$ and \(N+1\le   n(L+1)\). Hence
\begin{equation*} 
V^{-\lambda}
=
\left( m_{N+1}/Y_{N+1}\right)^\lambda\leq\left( m_{N+1}/[N+1]_q \right)^\lambda
\le
C_{\lambda,L} (1-q)^{-\lambda},
\end{equation*}
where we used $Y_{N+1}=[N+1]_q q^{-X(N+1)}\ge [N+1]_q$.
Since \(\{V<1/2\}\subseteq\{|V-1|>1/2\}\), Lemma
\ref{lem:concentration M} gives
\begin{equation*} 
\mathbb P(V<1/2)
\le
2^p\mathbb E|V-1|^p
\le
C_{L,p}(1-q)^{p/2}.
\end{equation*}
Combining the last two displayed equations, using \(p/2-\lambda>0\),
\[
\begin{aligned}
&\mathbb E\left[
\left|V^{-\lambda}-1\right|\mathbf 1_{\{V<1/2\}}
\right]
 \le
\mathbb E\left[
\left(V^{-\lambda}+1\right)\mathbf 1_{\{V<1/2\}}
\right] \\ & \le
C_{\lambda,L} (1-q)^{-\lambda}\mathbb P(V<1/2)
+\mathbb P(V<1/2)  \le
C_{\lambda,L,p}
\left((1-q)^{p/2-\lambda}+(1-q)^{p/2}\right) .
\end{aligned}
\]
Combining with \eqref{eq:VN+1 power lambda proof5}, we conclude the proof, since the right sides converge to $0$ as $n\to\infty$.
\end{proof} 

The above results allow us now to derive a general approximation of the Bryc polynomials under a weakly asymmetric scaling.

\begin{lemma}\label{lem:sup Bryc}
        Fix \(L>0\) and \(\lambda\in\mathbb R\). For each \(n\in\mathbb N\), let \(q=e^{-2/\sqrt n}\) and \(N=\lfloor nL\rfloor\). Consider the Bryc polynomial
\(W_{\bs}\) evaluated with parameter \(q\). For a composition $\bs$ of $N+1$, write $M_{\bs}(q)=\sum_{\ell=0}^r \sigma_\ell q^\ell $. Then 
\[
\lim_{n\to\infty}\sup_{\bs}
\left|
\left((1-q)M_{\bs}\left(q\right)\right)^\lambda W_{\bs}\left(q^\lambda\right)-1
\right|=0,
\]where the supremum is over all compositions \(\bs\) of \(N+1\). 
\end{lemma}
\begin{proof} We assume \(N\ge1\).
    By Theorem
\ref{thm:prob interpretation of Bryc polynomials},
$W_{\bs}(q^\lambda)=\mathbb E\left[q^{\lambda X^{\boldsymbol{\mathsf S}^{\bs}}(N+1)}\right]$.
Recall that $Y_{i}=[i]_q q^{-X^{\boldsymbol{\mathsf S}^{\bs}}(i)}$, and $V_{i}=Y_{i}/\mathbb{E}[Y_{i}]$ for $i\in\llbracket1,N+1\rrbracket$.  By Lemma \ref{lem:Y chain},  $\mathbb E [Y_{N+1}]=M_{\bs}(q)$, hence
\[  
\left((1-q)M_{\bs}(q)\right)^\lambda W_{\bs}\left(q^\lambda\right)
=
\left(1-q^{N+1}\right)^\lambda
\mathbb E\left[V_{N+1}^{-\lambda}\right].
\]
Therefore, 
 \[
 \begin{aligned}
&\sup_{\bs}
\left|
\left((1-q)M_{\bs}\left(q\right)\right)^\lambda W_{\bs}\left(q^\lambda\right)-1
\right|=
\sup_{\bs}
\left|\left(1-q^{N+1}\right)^\lambda
\mathbb E\left[
V_{N+1}^{-\lambda}
\right]-1\right|\\
&\le
\left(1-q^{N+1}\right)^\lambda\sup_{\bs}
\left|
\mathbb E\left[
V_{N+1}^{-\lambda}
\right]-1\right|
+
\left|\left(1-q^{N+1}\right)^\lambda-1\right|.
\end{aligned}\]
In view of Lemma \ref{lem:VN+1 power lambda} and $q^{N+1}\to0$, the right side above converges to $0$ as $n\to\infty$.
\end{proof}

We have now all tools to derive the desired approximation of the Bryc polynomial for compositions originating from a path contained in $\lrw(n,L,R)$.

\begin{proof}[Proof of Proposition \ref{prop:weak asym Bryc}]
Write \(q=q_n=e^{-2/\sqrt n}\) and $N=\lfloor nL\rfloor$ for \(n\in\mathbb N\). For a composition $\bs$ of $N+1$, recall that $M_{\bs}(q)=\sum_{\ell=0}^r \sigma_\ell q^\ell $. For \(\bg\in\lrw_N\), let
\[
\mathcal{S}(\bg)=\sum_{j=0}^{N}q^{\gamma_j},\qquad B(\bg)=\left((1-q)M_{\bs(\bg)}(q)\right)^\lambda W_{\bs(\bg)}\left(q^\lambda\right).
\]
We claim that
\begin{equation}\label{eq:weakBryc-pathidentity}
\left(2\sqrt n\right)^\lambda q^{-\lambda\min(\bg)}
W_{\bs(\bg)}\left(q^\lambda\right)
=
B(\bg)
\left(
(1-q)/\left(2\sqrt n\right)\cdot \mathcal{S}(\bg)
\right)^{-\lambda}.
\end{equation}
Indeed, this follows from
\[ 
M_{\bs(\bg)}(q)
 =
\sum_{\ell=0}^{\max(\bg)-\min(\bg)}
\#\{j:\gamma_j=\min(\bg)+\ell\}\,q^\ell   
 =
q^{-\min(\bg)}
\sum_{j=0}^{N}q^{\gamma_j}
=
q^{-\min(\bg)}\mathcal{S}(\bg). 
\]
For clarity, we split the rest of the proof into two steps.\\

\noindent\textbf{Step 1.} In this step, we prove the first statement that for fixed $R\geq0$,
\begin{equation}\label{eq:firststatementbsrthbrst}
\lim_{n\to\infty}\sup_{\bg\in\lrw(n,L,R)}
\left|
\left(2\sqrt n\right)^\lambda
q^{-\lambda\min(\bg)}
W_{\bs(\bg)}\left(q^\lambda\right)
-
\left(    \mathcal{S}(\bg)/n\right)^{-\lambda}
\right|=0 ,
\end{equation}
where we recall
\[
\lrw(n,L,R)=
\left\{
\bg\in\lrw_{\lfloor nL\rfloor}:
\left| \gamma_{\lfloor nx\rfloor}\right|
\le Rx\sqrt n\mbox{ for every }x\in[0,L]
\right\}.
\]

We first prove the claim that
\begin{equation}\label{eq:weakBryc-Inbounds}
    L e^{-2LR}
\le
  \mathcal{S}(\bg)/n
\le
(L+1)e^{2LR}\quad\mbox{for }\bg\in\lrw(n,L,R).
\end{equation}
Indeed, for $\bg\in\lrw(n,L,R)$ and $j\in\llbracket0,N\rrbracket$, since $j/n\in [0,L]$,  we have $|\gamma_j|\le R (j/n)\sqrt n\le RL\sqrt n $, and hence  
$e^{-2LR}\le q^{\gamma_j}
=e^{- 2\gamma_j/\sqrt n}{}
\le e^{2LR}$. Combining with $(N+1)/n\in[L,L+1]$, we conclude  \eqref{eq:weakBryc-Inbounds}. 

We next prove the claim that
\begin{equation}\label{eq:weakBryc-powercompare}
\lim_{n\to\infty}\sup_{\bg\in\lrw(n,L,R)}
\left|
\left(
(1-q)/\left(2\sqrt n\right)\cdot \mathcal{S}(\bg)
\right)^{-\lambda}
-
\left(
 \mathcal{S}(\bg)/n
\right)^{-\lambda}
\right|=0.
\end{equation}
Indeed, Taylor's formula applied to \(e^{-x}\) at \(x=0\), with \(x=2/\sqrt n\), implies
$(1-q)/(2\sqrt n)=(1-e^{-2/\sqrt n})/(2\sqrt n)=1/n+O(n^{-3/2})$. Together with \(\mathcal{S}(\bg)\le n(L+1)e^{2LR}\), this implies
\[
\lim_{n\to\infty}\sup_{\bg\in\lrw(n,L,R)}
\left|
(1-q)/\left(2\sqrt n\right)\cdot \mathcal{S}(\bg)
-
\mathcal{S}(\bg)/n
\right|=0.
\]
Since \((1-q)\sqrt n/2\to1\), \eqref{eq:weakBryc-Inbounds} implies that, for all large enough \(n\)
and all \(\bg\in\lrw(n,L,R)\),  $(1-q)/ (2\sqrt n )\cdot \mathcal{S}(\bg)$ and $\mathcal{S}(\bg)/n$ lie in a common compact interval \(K_{L,R}\subseteq(0,\infty)\).  Note \(x\mapsto x^{-\lambda}\) is uniformly continuous and bounded on \(K_{L,R}\),  we conclude \eqref{eq:weakBryc-powercompare}.
By Lemma \ref{lem:sup Bryc},
\[
\limsup_{n\to\infty}\sup_{\bg\in\lrw_N}
\left|B(\bg)-1\right|
\le
\lim_{n\to\infty}\sup_{\bs}
\left|
\left((1-q)M_{\bs}\left(q\right)\right)^\lambda W_{\bs}\left(q^\lambda\right)-1
\right|=0,
\]
where the supremum  is over all compositions of \(N+1\). Hence
$\lim_{n\to\infty}\sup_{\bg\in\lrw_N}\left|B(\bg)-1\right|=0 $.
Combining with \eqref{eq:weakBryc-pathidentity}, \eqref{eq:weakBryc-powercompare}, and
using the boundedness of
\(\left((1-q)/(2\sqrt n)\cdot \mathcal{S}(\bg)\right)^{-\lambda}\) on
\(K_{L,R}\), we conclude \eqref{eq:firststatementbsrthbrst}.\\

\noindent\textbf{Step 2.} In this step, we prove the second statement that there exists a constant \(C<\infty\) depending only on \(L\)
and \(\lambda\), such that for every \(R\ge0\), every \(n\in\mathbb N\),
and every \(\bg\in\lrw(n,L,R)\),
\begin{equation}\label{eq:firststatementbsrthbrst2}
\left(2\sqrt n\right)^\lambda
q^{-\lambda\min(\bg)}
W_{\bs(\bg)}\left(q^\lambda\right)
\le
Ce^{CR}.
\end{equation}
By Lemma \ref{lem:sup Bryc}, there exists \(n_0=n_0(L,\lambda)\) such that $\sup_{\bg\in\lrw_N}|B(\bg)-1|\le 1$
for all $n\ge n_0$.
For \(n<n_0\), there are only finitely many possible \((n,\bs)\), and Theorem \ref{thm:prob interpretation of Bryc polynomials} gives
\(W_{\bs}(q^\lambda)=\mathbb E[q^{\lambda X^{\boldsymbol{\mathsf S}^{\bs}}(N+1)}]<\infty\). Note \(M_{\bs}(q)>0\), enlarging the constant,  
  $B(\bg)\le C_{\lambda,L}$ for every $n\in\mathbb N$ and  $\bg\in\lrw_N$.
Combining with \eqref{eq:weakBryc-pathidentity},
\begin{equation}\label{eq:weakBryc-pathbound}
\left(2\sqrt n\right)^\lambda q^{-\lambda\min(\bg)}
W_{\bs(\bg)}\left(q^\lambda\right)
\le
C_{\lambda,L}
\left(
(1-q)/(2\sqrt n)\cdot \mathcal{S}(\bg)
\right)^{-\lambda}\quad\mbox{for all }\bg\in\lrw_N.
\end{equation}
Since \(q=e^{-2/\sqrt n}\), the inequalities \(1-e^{-x}\le x\) and
\(1-e^{-x}\ge e^{-2}x\) for \(x\in[0,2]\), applied with
\(x=2/\sqrt n\), give
$e^{-2}\le (1-q)\sqrt n/2\le 1 $.
Combining this with \eqref{eq:weakBryc-Inbounds}, for every
\(\bg\in\lrw(n,L,R)\),
\[
e^{-2}L e^{-2LR}
\le
(1-q)/(2\sqrt n)\cdot \mathcal{S}(\bg)
\le
(L+1)e^{2LR}.
\]
Therefore,
\[
\left((1-q)/(2\sqrt n)\cdot \mathcal{S}(\bg)\right)^{-\lambda}
\le C_{\lambda,L}e^{2|\lambda|LR}.
\]
Plugging this into \eqref{eq:weakBryc-pathbound} and enlarging the constant $C_{\lambda,L}$ if necessary gives \eqref{eq:firststatementbsrthbrst2}. 
\end{proof}

\section{Weakly asymmetric scaling limits for the stationary distribution of the open ASEP}\label{sec:MainProof}

We have now all tools in order to show our main result Theorem~\ref{thm:open KPZ equation scaling} on the stationary distribution of the open ASEP under a weakly asymmetric scaling, which we restate here for clarity.

\begin{theorem}\label{thm:open KPZ equation scaling again}
Fix \(L>0\) and \(u,v\in\mathbb R\).  For \(n\in\mathbb N\), set
\[
N_n=\lfloor nL\rfloor,
\qquad
q_n=e^{-2/\sqrt n},
\qquad
\a_n=q_n^u,
\qquad
\b_n=q_n^v .
\]
Define 
\[
\alpha_n:=\frac{1}{1+\a_n},
\qquad
\beta_n:=\frac{1}{1+\b_n},
\qquad
\gamma_n:=q_n(1-\alpha_n),
\qquad
\delta_n:=q_n(1-\beta_n).
\]
For all sufficiently large \(n\), so that \(N_n\ge1\), let
\(\tau=(\tau_1,\ldots,\tau_{N_n})\sim
\mu_{N_n}^{(q_n,\a_n,\b_n)}\).  For \(0\le j\le N_n\) and \(0\le x\le L\), 
set
\begin{equation}\label{eq:height functions}
H_n^\tau(j)=\sum_{i=1}^j\tau_i,
\qquad
h_n^\tau(x)
=
\frac{2H_n^\tau(\lfloor nx\rfloor)-\lfloor nx\rfloor}{\sqrt n}.
\end{equation}
We view \(h_n^\tau\) as a random element of \(D([0,L],\mathbb R)\) equipped
with the Skorokhod \(J_1\) topology.  Then, as \(n\to\infty\),
$h_n^\tau
\Rightarrow
B_1$
weakly in \(D([0,L],\mathbb R)\), where \(B_1\) is the first coordinate under
\(\mathbb P_{\mathrm{KPZ}}^{u,v}\) from Definition \(\ref{def:KPZstat}\).
\end{theorem}

\begin{proof}
To simplify notation, we write in the following \(N=N_n\) and \(q=q_n\). Since finite values of \(n\) do not affect the weak convergence, we can assume $N\geq1$. 
Let $S_1=\{S_1(j)\}_{j=0}^N$, $S_2=\{S_2(j)\}_{j=0}^N$ be independent simple symmetric random walks started at $S_1(0)=S_2(0)=0$. Let
$\bg=(\gamma_0,\dots,\gamma_N)=\left\{(S_1(j)-S_2(j))/2\right\}_{j=0}^N$.
Then \(\bg\in\lrw_N\) and 
$\mathbb P\left((S_1-S_2)/2=\bg\right)=4^{-N}2^{H(\bg)}$ for fixed \(\bg\in\lrw_N\).
Consider the functional $\Phi:\lrw_N\to(0,\infty)$,
\[
\Phi(\bg)
:=
\b_n^{\gamma_N}
(\a_n\b_n)^{-\min(\bg)}
W_{\bs(\bg)}(\a_n\b_n).
\] 
This quantity is finite and strictly positive as by Theorem \ref{thm:prob interpretation of Bryc polynomials},
$W_{\bs(\bg)}(q^{u+v})
=
\mathbb E\left[q^{(u+v)X^{\boldsymbol{\mathsf S}^{\bs(\bg)}}(N+1)}\right]\in(0,\infty)$,
since \(q^{u+v}>0\) and \(X^{\boldsymbol{\mathsf S}^{\bs(\bg)}}(N+1)\) takes only finitely many values.
Define the probability measure \(\mathbb Q\) by
\[
\frac{d \mathbb Q}{d\mathbb P}(S_1,S_2)\propto\Phi\left((S_1-S_2)/2\right).
\]
Under \(\mathbb Q\), the induced law of \(\bg=(S_1-S_2)/2\) is \(\mathbb P^{(q,\a_n,\b_n)}\) from \eqref{eq:PNab}.  Since the Radon--Nikodym derivative depends only on \(\bg\), conditional on \(\bg\) the choices of the \(S_1\)-increments on the horizontal steps of \(\bg\) remain independent fair signs. Therefore Theorem \ref{thm:characterization stationary measure ASEP} shows that
\(\tau=(\tau_1,\dots,\tau_N)\in\{0,1\}^N\), defined by
\[
\tau_j= (1+S_1(j)-S_1(j-1))/2,\qquad j\in\llbracket1,N\rrbracket,
\]
has law \(\mu_{N}^{(q_n,\a_n,\b_n)}\).
Let $D_{1}(x)=S_1(\lfloor nx\rfloor)/\sqrt n$ and $D_{2}(x)=S_2(\lfloor nx\rfloor)/\sqrt n$ for $x\in[0, L]$. 
Then the height functions \eqref{eq:height functions} are
\begin{equation}\label{eq:height again}
H_n^\tau(j)= (j+S_1(j))/2,
\qquad
h_n^\tau(x)=D_1(x).
\end{equation}

Define
\[
\Psi_n\left(D_1,D_2\right)=(2\sqrt n)^{u+v} q^{-(u+v)\min(\bg)} e^{-v(D_1(L)-D_2(L)) } W_{\bs(\bg)}(q^{u+v}),
\]
where $\bg=(S_1-S_2)/2\in\lrw_N$.  Since
\(\Psi_n(D_1,D_2)=(2\sqrt n)^{u+v}\Phi(\bg)\), the change of measure is
\[
\frac{d \mathbb Q}{d\mathbb P}(S_1,S_2)
=
\frac{\Psi_n\left(D_1,D_2\right)}
{\mathbb E_{\mathbb P}\left[\Psi_n\left(D_1,D_2\right)\right]} , 
\]
where we write $\mathbb E_{\mathbb P}$ for the expectation with respect to $\mathbb P$. 
We write $J_n(f,g)=\frac1n\sum_{j=0}^{N }e^{-(f(j/n)-g(j/n))}$ for $f,g\in D([0,L],\mathbb R)$.
Define the functional
\[
\Theta_n\left(D_1,D_2\right)
=
e^{-v(D_1(L)-D_2(L))}J_n(D_{1},D_{2})^{-(u+v)} .
\]
Let 
$T_n(D_1,D_2)=\sup_{0\le x\le L}|D_1(x)-D_2(x)|$.
For clarity, we split the rest of the proof into three steps.\\

\noindent\textbf{Step 1.} In this step, we conclude the proof of the theorem assuming 
\begin{align}
    &\label{eq:claim exponential moments}\sup_n\mathbb E_{\mathbb P}\left[e^{\theta T_n}\right]<\infty
\qquad\mbox{for every }\theta<\infty ,\\
&\label{eq:phithetaclose}\lim_{n\to\infty}\mathbb E_{\mathbb P}\left[
\left|\Psi_n(D_1,D_2)-\Theta_n(D_1,D_2)\right|
\right]=0 .
\end{align} 
By Donsker's theorem, under \(\mathbb P\), $(D_1,D_2)\Rightarrow(B_1, B_2)$ in \(D([0,L],\mathbb R)^2\), where \(B_1, B_2\) are
independent standard Brownian motions. 
Note that, if \(f_n\to f\) and \(g_n\to g\) in \(D([0,L],\mathbb R)\) and \(f,g\) are continuous, then the \(J_1\)-convergence is uniform on \([0,L]\). Hence \(e^{-(f_n-g_n)}\to e^{-(f-g)}\) uniformly, and
\[
\lim_{n\to\infty}\frac1n\sum_{j=0}^{\lfloor nL\rfloor}
e^{-(f_n(j/n)-g_n(j/n))}
=
\int_0^L e^{-(f(s)-g(s))}\,ds ,
\]
because the difference from the corresponding Riemann sum for the continuous function \(e^{-(f-g)}\) tends to \(0\), and that Riemann sum converges to the displayed integral.  The endpoint evaluations also converge, again because the \(J_1\)-convergence to a continuous limit is uniform.  By the extended continuous mapping theorem, 
\[
(D_1,D_2,\Theta_n)
\Rightarrow
(B_1, B_2,\Lambda) \quad \text{where } \
\Lambda
=
e^{-v(B_1(L)-B_2(L))}
\left(
\int_0^L e^{-(B_1(s)-B_2(s))}\,ds
\right)^{- u-v}.
\]
The bound \eqref{eq:claim exponential moments} implies that \(\{\Theta_n\}_{n\ge1}\) is uniformly integrable.  Indeed, we have that
\[
L e^{-T_n}\le J_n(D_1,D_2)\le (L+1)e^{T_n},
\qquad
|D_1(L)-D_2(L)|\le T_n,
\]
and hence, for a constant \(C<\infty\) depending only on \(L,u,v\), we see $0\le \Theta_n(D_1,D_2)\le C e^{(|u|+2|v|)T_n}$.
\[
0\le \Theta_n(D_1,D_2)\le C e^{(|u|+2|v|)T_n}.
\]
Choosing \(\theta>|u|+2|v|\) in \eqref{eq:claim exponential moments} gives
\[
\sup_n\mathbb E_{\mathbb P}\left[
e^{(|u|+2|v|)T_n}\mathbf 1_{\{T_n>R\}}
\right]
\le
e^{-(\theta-|u|-2|v|)R}
\sup_n\mathbb E_{\mathbb P}\left[e^{\theta T_n}\right]
\to0,
\]
so \(\{\Theta_n\}_{n\ge1}\) is uniformly integrable.  If \(F:D([0,L],\mathbb R)\to\mathbb R\) is bounded and continuous, then the continuous mapping theorem applied to \((D_1,D_2,\Theta_n)\) gives \(F(D_1)\Theta_n\Rightarrow F( B_1)\Lambda\).  Moreover, \(\{F(D_1)\Theta_n\}_{n\ge1}\) is uniformly integrable, since \(|F(D_1)\Theta_n|\le \|F\|_\infty\Theta_n\).  Hence, by convergence in
distribution together with uniform integrability,
\[
\lim_{n\to\infty}\mathbb E_{\mathbb P}\left[F(D_1)\Theta_n(D_1,D_2)\right]
=
\mathbb E_{\mathrm{BM}}^{0,0}\left[F( B_1)\Lambda\right].
\]
Since \[
\left|
\mathbb E_{\mathbb P}\left[F(D_1)\Psi_n(D_1,D_2)\right]
-
\mathbb E_{\mathbb P}\left[F(D_1)\Theta_n(D_1,D_2)\right]
\right|
\le
\|F\|_\infty
\mathbb E_{\mathbb P}\left[
\left|\Psi_n(D_1,D_2)-\Theta_n(D_1,D_2)\right|
\right]
\to0,
\]
we have
\[
\lim_{n\to\infty}\mathbb E_{\mathbb P}\left[F(D_1)\Psi_n(D_1,D_2)\right]=\mathbb E_{\mathrm{BM}}^{0,0}\left[F( B_1)\Lambda\right].
\]
Taking \(F\equiv1\),
\[
\lim_{n\to\infty}\mathbb E_{\mathbb P}\left[\Psi_n(D_1,D_2)\right]
=
\mathbb E_{\mathrm{BM}}^{0,0}[\Lambda]
=
\mathcal Z_{\mathrm{KPZ}}^{u,v}.
\] 
In view of \eqref{eq:height again},  for every bounded continuous \(F:D([0,L],\mathbb R)\to\mathbb R\),
\[ 
\mathbb E\left[F(h_n^\tau)\right]
=
\mathbb E_{\mathbb Q}\left[F(D_1)\right]  
=
\frac{
\mathbb E_{\mathbb P}\left[F(D_1)\Psi_n(D_1,D_2)\right]
}{
\mathbb E_{\mathbb P}\left[\Psi_n(D_1,D_2)\right]
} 
\to
\frac{
\mathbb E_{\mathrm{BM}}^{0,0}\left[F( B_1)\Lambda\right]
}{
\mathcal Z_{\mathrm{KPZ}}^{u,v}
}
=
\mathbb E_{\mathrm{KPZ}}^{u,v}\left[F(B_1)\right].
\]
This proves \(h_n^\tau\Rightarrow B_1\) in \(D([0,L],\mathbb R)\).\\

\noindent\textbf{Step 2.} 
In this step, we prove \eqref{eq:claim exponential moments}, i.e.~$
\sup_n\mathbb E_{\mathbb P}\left[e^{\theta T_n}\right]<\infty$
for every $\theta<\infty$.
Indeed, let \(M_n(S)=\max_{0\le j\le N}|S(j)|/\sqrt n\) for a simple symmetric random walk \(S\).  For \(a>0\), Doob's maximal inequality applied to the nonnegative submartingales \(e^{aS(j)}\) and \(e^{-aS(j)}\), together with
\[
\mathbb E e^{aS(N)}=(\cosh a)^N\le e^{a^2N/2},
\]
gives, for \(t>0\), using the reflection principle for the symmetric simple random walk and the exponential Markov inequality,
\[
\mathbb P\left(M_n(S)\ge t\right)
\le
2\inf_{a>0}\exp\left(-at\sqrt n+a^2N/2\right)
\le
2e^{-t^2/(2L)} ,
\]
where \(N=\lfloor nL\rfloor\le nL\).  Therefore, for every \(\theta>0\),
\[
\sup_n \mathbb E \left[e^{\theta M_n(S)}\right]
\le
1+2\theta\int_0^\infty e^{\theta t-t^2/(2L)}\,dt<\infty .
\]
Since
\[
T_n(D_1,D_2)\le M_n(S_1)+M_n(S_2),
\]
and \(S_1,S_2\) are independent, we conclude for \(\theta>0\); for \(\theta\le0\) this follows from \(T_n\ge0\).\\

\noindent\textbf{Step 3.}
In this step, we prove the claim \eqref{eq:phithetaclose}, which we restate as
\begin{equation}\label{eq:phithetaclose again}
\lim_{n\to\infty}\mathbb E_{\mathbb P}\left[
\left|\Psi_n(D_1,D_2)-\Theta_n(D_1,D_2)\right|
\right]=0 .
\end{equation}
For \(\bg=(S_1-S_2)/2\in\lrw_N\), write
\[
S(\bg)=\sum_{j=0}^N q^{\gamma_j},
\qquad
A_n(\bg)=
(2\sqrt n)^{u+v} q^{-(u+v)\min(\bg)}
W_{\bs(\bg)}(q^{u+v}).
\]
Then \(J_n(D_1,D_2)=S(\bg)/n\) and
\[
\Psi_n(D_1,D_2)=e^{-v(D_1(L)-D_2(L))}A_n(\bg),
\qquad
\Theta_n(D_1,D_2)=e^{-v(D_1(L)-D_2(L))}J_n(D_1,D_2)^{-u-v}.
\]
We write \(M_{\bs}(q)=\sum_{\ell=0}^r\sigma_\ell q^\ell\) for a composition \(\bs\) of \(N+1\). Then $M_{\bs(\bg)}(q)=q^{-\min(\bg)}S(\bg)$. Define 
\[
C_n(\bg)=
\left((1-q)M_{\bs(\bg)}(q)\right)^{u+v}
W_{\bs(\bg)}(q^{u+v}),
\]
then 
\begin{equation}\label{eq:yjmfuujkf1}
A_n(\bg)=C_n(\bg)\left((1-q)/(2\sqrt n)\cdot S(\bg)\right)^{-u-v}=C_n(\bg)\left(\left( (1-q)\sqrt n/2 \right)\cdot J_n(D_1,D_2)\right)^{-u-v}.
\end{equation}
By Lemma \ref{lem:sup Bryc}, 
\begin{equation}\label{eq:yjmfuujkf2}
    \lim_{n\to\infty}\sup_{\bg\in\lrw_N}|C_n(\bg)-1|=0,\qquad \sup_{n}\sup_{\bg\in\lrw_N} C_n(\bg)<\infty  ,
\end{equation} 
where the second claim follows from the first one, since there exists \(n_0\) such that \(\sup_{\bg\in\lrw_N}C_n(\bg)\le2\) for all \(n\ge n_0\), and for \(n<n_0\), there are only finitely many possible paths \(\bg\).
For each fixed \(R<\infty\),  
\[
L e^{-R}\le J_n(D_1,D_2)\le (L+1)e^R
\] on the event $\{T_n\leq R\}$.
Combining the last display and \eqref{eq:yjmfuujkf1}, \eqref{eq:yjmfuujkf2}, in view of \((1-q)\sqrt n/2\to1\),
\[
\left|A_n(\bg)-J_n(D_1,D_2)^{-u-v}\right|\to0
\]
uniformly on $\{T_n\leq R\}$. Note  that on \(\{T_n\leq R\}\), the factor \(e^{-v(D_1(L)-D_2(L))}\) is bounded by \(e^{|v|R}\), and so
\begin{equation}\label{eq:yjmfuujkf3}
\lim_{n\to\infty}\mathbb E_{\mathbb P}\left[\left|\Psi_n(D_1,D_2)-\Theta_n(D_1,D_2)\right|\mathbf 1_{\{T_n\leq R\}}\right]=0
\end{equation} for every $R<\infty$.
It remains to control the expectation on the complement.
Since \(N=\lfloor nL\rfloor\),
\[
L e^{-T_n}
\le
J_n(D_1,D_2)
\le
(L+1)e^{T_n}.
\]
Since \(e^{-2}\le (1-q)\sqrt n/2\le1\), \eqref{eq:yjmfuujkf1}
and \eqref{eq:yjmfuujkf2}  imply, for a constant \(C<\infty\) depending only on \(L,u,v\),
\[
0\le A_n(\bg)\le C e^{|u+v|T_n},
\qquad
0\le J_n(D_1,D_2)^{-u-v}\le C e^{|u+v|T_n}.
\]
Since \(|D_1(L)-D_2(L)|\le T_n\), it follows that
 \[
0\le \Psi_n(D_1,D_2)+\Theta_n(D_1,D_2)
\le
 C e^{(|u|+2|v|)T_n}.
 \] 
  Choose \(\theta>|u|+2|v|\), then \eqref{eq:claim exponential moments} gives
  \[
  \sup_n
 \mathbb E_{\mathbb P}\left[
  e^{(|u|+2|v|)T_n}\mathbf 1_{\{T_n>R\}}
 \right]
  \le
  e^{-(\theta-(|u|+2|v|))R}
  \sup_n\mathbb E_{\mathbb P}\left[e^{\theta T_n}\right].
 \]
  Therefore,
 \begin{equation}\label{eq:yjmfuujkf5}
  \lim_{R\to\infty}\sup_n
  \mathbb E_{\mathbb P}\left[
  e^{(|u|+2|v|)T_n}\mathbf 1_{\{T_n>R\}}
  \right]=0.
  \end{equation}
  Combining \eqref{eq:yjmfuujkf3} and \eqref{eq:yjmfuujkf5} concludes the proof of \eqref{eq:phithetaclose again}.
\end{proof}

\begingroup
\let\clearpage\relax
\let\newpage\relax
\bibliographystyle{goodbibtexstyle.bst}
\bibliography{RWApproximation}
\endgroup

\end{document}